\documentclass[12pt,twoside,reqno]{amsart}
\pdfoutput=1
\usepackage{amsmath, amssymb, amsthm, enumitem, esint}
\usepackage{xcolor}
\usepackage[
backend=biber,
style=numeric
]{biblatex}
\DeclareNameAlias{default}{family-given}
\usepackage[hyperindex]{hyperref}
\usepackage{tikz}
\usepackage{tkz-euclide}
\usetikzlibrary{decorations.pathreplacing,shapes.misc,calc, patterns,through}
\usepackage{pgfplots}
\usepackage{enumitem}
\usepgfplotslibrary{fillbetween}
\pgfplotsset{compat=1.11}

\usepackage{hyperref}
\usepackage{fancyhdr}
\hypersetup{bookmarksdepth=3}

\newcommand{\EMPTY}[1]{}

\newtheorem{Theorem}[subsection]{Theorem}
\newtheorem{Lemma}[subsection]{Lemma}
\newtheorem{Corollary}[subsection]{Corollary}
\newtheorem{Proposition}[subsection]{Proposition}

\theoremstyle{definition}
\newtheorem{Definition}[subsection]{Definition}
\theoremstyle{remark}
\newtheorem{Remark}[subsection]{Remark}

\numberwithin{equation}{section}
\begin{document}
\pagestyle{fancy}
\fancyhf{} 
\renewcommand{\headrulewidth}{0pt}
\fancyhf[EHC]{B. Harvie}
\fancyhf[OHC]{Weak IMCF and Minkowski Inequalities in Hyperbolic Space}
\fancyhf[FC]{\thepage}

\title{On Weak Inverse Mean Curvature Flow and Minkowski-type Inequalities in Hyperbolic Space}
\author{Brian Harvie}
\address{Copenhagen Center for Geometry and Topology}
\email{brdh@math.ku.dk}

\date{\today}

\begin{abstract}
We prove that a proper weak solution $\{ \Omega_{t} \}_{0 \leq t < \infty}$ to inverse mean curvature flow in hyperbolic space $\mathbb{H}^{n}$, $3 \leq n \leq 7$ , is eventually smooth and star-shaped for an arbitrary initial domain $\Omega_{0}$. In fact, this happens by the time

\begin{equation*}
    T= (n-1) \log \left( \frac{\text{sinh} \left( r_{+} \right)}{ \text{sinh} \left( r_{-} \right)} \right),
\end{equation*}
where $r_{+}$ and $r_{-}$ are the geodesic out-radius and in-radius of $\Omega_{0}$. The approach is based on an Alexandrov reflection method for extrinsic curvature flows originally introduced by Chow-Gulliver \cite{CG01}. In addition, our methods characterize expanding spheres as proper weak IMCF on $\mathbb{H}^{n} \setminus \{ 0 \}$ for arbitrary $n$, thereby implying a result for ancient smooth solutions.

As applications of the regularity theorem, we derive optimal Minkowski inequalities for arbitrary smooth domains of $\mathbb{H}^{n}$, $3 \leq n \leq 7$, extending those of Brendle-Hung-Wang \cite{BHW12} and De Lima-Girao \cite{LG12}. From this, we also extend the Riemannian Penrose inequality from \cite{LG12} to balanced asymptotically graphs over the exteriors of outer-minimizing domains in $\mathbb{H}^{n}$.
\end{abstract}
\maketitle
\section{Introduction}
Hyperbolic space is the Riemannian manifold of constant negative curvature given by

\begin{equation} \label{hyperbolic_space}
     (\mathbb{H}^{n},g_{\mathbb{H}^{n}}) =(\mathbb{S}^{n-1} \times (0,\infty), dr \otimes dr + \sinh^{2}(r) g_{\mathbb{S}^{n-1}}),
\end{equation}
where $g_{\mathbb{S}^{n-1}}$ is the unit sphere metric. This paper considers the inverse mean curvature flow (IMCF) of hypersurfaces in hyperbolic space. Given a closed smooth manifold $\Sigma^{n-1}$, a one-parameter family of $C^{\infty}$ embeddings $X: \Sigma^{n-1} \times [0,T) \rightarrow \mathbb{H}^{n}$ solves IMCF if

\begin{equation} \label{IMCF}
    \frac{\partial}{\partial t} X(x,t) = \frac{1}{H} \nu(x,t), \hspace{2cm} (x,t) \in \Sigma^{n-1} \times [0,T),
\end{equation}
where $\nu$ is the outward unit normal and $H>0$ the mean curvature of the hypersurface $\Sigma_{t}=X_{t}(\Sigma^{n-1}) \subset \mathbb{H}^{n}$. IMCF enjoys a wide variety of geometric applications, but a key obstacle to these is the formation of singularities which terminate the flow in finite time. In \cite{HI99}, Huisken and Ilmanen introduced a notion of weak solutions of IMCF which flow beyond singularities. Given a bounded domain $\Omega_{0} \subset \mathbb{H}^{n}$ with connected $C^{\infty}$ boundary $\partial \Omega_{0}= \Sigma_{0}$, a proper weak solution to IMCF with initial condition $\Omega_{0}$ corresponds to a proper locally Lipschitz function $u: \mathbb{H}^{n} \rightarrow \mathbb{R}$ which solves the degenerate-elliptic exterior Dirichlet problem

\begin{eqnarray} \label{dirichlet}
    \text{div} \left( \frac{\nabla u}{|\nabla u} \right) &=& |\nabla u| \hspace{1cm} \text{in} \hspace{0.5cm} \mathbb{H}^{n} \setminus \overline{\Omega_{0}},  \\
    u &=& 0 \hspace{1.4cm} \text{on} \hspace{0.5cm} \Sigma_{0} \nonumber
\end{eqnarray}
in a variational sense, with weak flow surfaces $\Sigma_{t}$ defined by $\Sigma_{t}= \partial \{ u < t \}$.

For a smooth solution $u$ to \eqref{dirichlet},  $\Sigma_{t}$ are smooth, strictly mean-convex hypersurfaces and form a solution to \eqref{IMCF}, but for a general weak solution $\Sigma_{t}$ are only $C^{1,\alpha}$ with a possible singular set in higher dimensions. Furthermore, a phenomenon may occur where $\Sigma_{t}$ ``jumps" outward at certain times-- we discuss this more thoroughly in Section 2. Understanding the long-time behavior of proper weak IMCF is crucial to its geometric applications. For weak IMCF in $\mathbb{R}^{n}$, a blow-down lemma due to Huisken and Ilmanen (Lemma 7.1 in \cite{HI99}) ensures that the weak flow surfaces $\Sigma_{t}$ converge to a round sphere in a $C^{1}$ sense after scaling as $t \rightarrow \infty$. No such blow-down lemma applies to IMCF in $\mathbb{H}^{n}$, and an explicit example from \cite{HW14} shows that even smooth IMCF in $\mathbb{H}^{n}$ may not become round as $t \rightarrow \infty$, see also \cite{N10}. For weak IMCFs in $\mathbb{H}^{n}$, the asymptotic profile of $\Sigma_{t}$, including even the eventual topology, is unclear a priori.

In this paper, we show that $\Sigma_{t}$ are star-shaped after an explicit time $T$ determined by the in-radius and out-radius of the initial data. In dimensions less than $8$, this implies that $\Sigma_{t}$ form a classical solution to \eqref{IMCF} for $t \in (T,\infty)$, meaning in particular that the flow is smooth and free of jumps beyond time $T$. Our regularity result follows earlier work on the regularity of weak IMCF in $\mathbb{R}^{n}$ by Huisken-Ilmanen \cite{HI08} and in asymptotically hyperbolic $3$-manifolds by Shi-Zhu \cite{SZ21}.

\begin{Theorem} [Regularity of Weak IMCF in Hyperbolic Space] \label{star_shape}
Let $u \in C^{0,1}_{\text{loc}}(\mathbb{H}^{n})$ be a proper weak solution to IMCF with initial condition $\Omega_{0}$, and say WLOG that $0 \in \Omega_{0}$. Call $r_{+}= \max_{\Sigma_{0}} r$, $r_{-}=\min_{\Sigma_{0}} r$ and define

\begin{equation} \label{waiting_time}
    T= (n-1) \log \left( \frac{\sinh(r_{+})}{\sinh(r_{-})} \right).
\end{equation}
Then for each $t \in (T,\infty)$, $\Sigma_{t}= \partial \{ u < t \}$ is star-shaped, i.e. $\Sigma_{t}= \{ (r_{t}(\theta),\theta) | \theta \in \mathbb{S}^{n-1} \} \subset \mathbb{H}^{n}$ for a function $r_{t}$ on $\mathbb{S}^{n-1}$. Moreover, the function $r_{t}$ is Lipschitz and, where differentiable, satisfies the gradient estimate

\begin{equation} \label{support_bound}
|Dr_{t}|_{\mathbb{S}^{n-1}} \leq \frac{\sinh(r_{+}) \sinh(r_{t})}{\left(\sinh^{2}(r_{t}) - \sinh^{2}(r_{+}) \right)^{\frac{1}{2}}}.  
\end{equation}
Finally, if $3 \leq n \leq 7$ then $\Sigma_{t}$ are a smooth solution to \eqref{IMCF} for $t \in (T,\infty)$.
\end{Theorem}

A key application of IMCF is to Minkowski inequalities, which are fundamental geometric inequalities in Riemannian geometry. Recall that the classical Minkowski inequality in Euclidean space for a convex body $\Omega_{0}^{n} \subset \mathbb{R}^{n}$ with boundary $\Sigma_{0}= \partial \Omega_{0}$ states

\begin{equation*}
    \frac{1}{(n-1)w_{n-1}} \int_{\Sigma_{0}} H d\sigma \geq \left(\frac{|\Sigma_{0}|}{w_{n-1}} \right)^{\frac{n-2}{n-1}}.
\end{equation*}
 In their breakthrough paper \cite{BHW12}, Brendle, Hung, and Wang proved a weighted Minkowski inequality for mean-convex and star-shaped hypersurfaces $\Sigma_{0}^{n-1} \subset \mathbb{H}^{n}$. Later in \cite{LG12}, De Lima and Girao proved a related Minkowski inequality for such hypersurfaces with lower bound depending on the area of $\Sigma_{0}$. Their respective inequalities read

\begin{eqnarray}
   \frac{1}{(n-1)w_{n-1}} \int_{\Sigma_{0}} f H d\sigma &\geq& \left( \frac{|\Sigma_{0}|}{w_{n-1}} \right)^{\frac{n-2}{n-1}} + \frac{n}{w_{n-1}} \int_{\Omega_{0}} f d\Omega, \label{baby_minkowski_1} \\
   \frac{1}{(n-1)w_{n-1}} \int_{\Sigma_{0}} f H d\sigma &\geq& \left( \frac{|\Sigma_{0}|}{w_{n-1}} \right)^{\frac{n-2}{n-1}} + \left( \frac{|\Sigma_{0}|}{w_{n-1}} \right)^{\frac{n}{n-1}}, \label{baby_minkowski_2}
\end{eqnarray}
where $f \in C^{\infty}(\mathbb{H}^{n})$ is the potential function

\begin{equation} \label{potential}
    f(r)= \cosh(r).
\end{equation}
The proofs of \eqref{baby_minkowski_1}-\eqref{baby_minkowski_2} utilize quantities which are monotone under smooth IMCF. The failure of blow-down arguments in $\mathbb{H}^{n}$ presents a key difficulty to proving these inequalities for larger classes of hypersurfaces using weak IMCF. However, Theorem \ref{star_shape} immediately overcomes this obstacle in low dimensions by ensuring that $\Sigma_{t}$ are eventually smooth, star-shaped, and strictly mean-convex.

Recall that a domain $\Omega_{0} \subset \mathbb{H}^{n}$ is \textit{outer-minimizing} if $|\partial^{*} F| \geq |\partial^{*} \Omega_{0}|$ for every bounded domain $F$ of finite perimeter $|\partial^{*} F|$ properly containing $\Omega_{0}$, and \textit{strictly outer-minimizing} if equality implies $F=\Omega$ up to a.e. equivalence. The \textit{strictly minimizing hull} $\Omega_{0}^{+}$ of $\Omega_{0}$ is the intersection of all bounded strictly outer-minimizing domains containing $\Omega_{0}$ and is itself strictly outer-minimizing. Exploiting regularity of weak IMCF in low dimensions, we establish versions of \eqref{baby_minkowski_1}-\eqref{baby_minkowski_2} for arbitrary smooth domains of $\mathbb{H}^{n}$ with lower bounds determined $\Omega^{+}_{0}$ (compare, e.g. with inequality (1.2) in \cite{BFM24}).

\begin{Theorem}[Minkowski Inequalities in Hyperbolic Space] \label{minkowski_ineq}
Let $\Omega_{0} \subset \mathbb{H}^{n}$, $3 \leq n \leq 7$, be a bounded domain with connected $C^{\infty}$ boundary $\partial \Omega_{0}= \Sigma_{0}$. Then we have the inequalities
\begin{eqnarray}
   \frac{1}{(n-1)w_{n-1}} \int_{\Sigma_{0}} |f H| d\sigma &\geq& \left( \frac{|\Sigma_{0}^{+}|}{w_{n-1}} \right)^{\frac{n-2}{n-1}} + \frac{n}{w_{n-1}} \int_{\Omega_{0}^{+}} f d\Omega, \label{minkowski} \\
   \frac{1}{(n-1)w_{n-1}} \int_{\Sigma_{0}} |f H| d\sigma &\geq& \left( \frac{|\Sigma_{0}^{+}|}{w_{n-1}} \right)^{\frac{n-2}{n-1}} + \left( \frac{|\Sigma_{0}^{+}|}{w_{n-1}} \right)^{\frac{n}{n-1}}, \label{minkowski_area}
\end{eqnarray}
on $\Sigma_{0}$, where $\Sigma_{0}^{+}=\partial \Omega_{0}^{+}$ is the boundary of the strictly minimizing hull of $\Omega_{0}$. In both cases, equality holds if and only if $\Sigma_{0} = \mathbb{S}^{n-1} \times \{ r_{0} \}$ is a slice of \eqref{hyperbolic_space}. Consequently, inequalities \eqref{baby_minkowski_1}-\eqref{baby_minkowski_2} hold provided that $\Omega_{0}$ is outer-minimizing.
\end{Theorem} 
The pure-area Minkowski inequality \eqref{minkowski_area} carries further implications for the conjectured Penrose inequality for asymptotically hyperbolic manifolds. Consider the following model for $(n+1)$-dimensional hyperbolic space:
\begin{equation}
    (\mathbb{H}^{n+1}, \overline{g}) = ( \mathbb{H}^{n} \times \mathbb{R}, f^{2}(r)d\tau^{2} + g_{\mathbb{H}^{n}}).
\end{equation}
In \cite{LG12} and following earlier work by Dahl, Gicquaud, and Sakovich in \cite{DGS12}, de Lima and Girao considered non-compact hypersurfaces $M^{n} \subset \mathbb{H}^{n+1}$ satisfying the following:

\begin{enumerate}[label=(\Alph*)]
    \item $M^{n} \subset  (\mathbb{H}^{n+1}, \overline{g})$ is asymptotically hyperbolic and balanced in the sense of \cite{LG12}.
    \item $M^{n}= \text{graph}(h)$ for a function $h \in C^{\infty}(\mathbb{H}^{n} \setminus \overline{\Omega_{0}})$ defined over the exterior of a smooth domain $\Omega_{0} \subset \mathbb{H}^{n}$ containing $0$.
    \item $\tau \equiv \tau_{0}$ on $\partial M^{n} \subset \mathbb{H}^{n+1}$, and $M^{n}$ intersects $\mathbb{H}^{n} \times \{ \tau_{0} \}$ orthogonally along $\partial M$.
\end{enumerate}
Note that the last item implies that $\partial M$ considered as a hypersurface in $(M^{n},\overline{g}|_{M^{n}})$ is totally geodesic. Balanced asymptotically hyperbolic manifolds $(M^{n},g)$ carry a natural mass-like invariant $m$, which is studied in \cite{DGS12}. When a balanced asympotically hyperbolic $(M^{n},g)$ satisfies (B) and (C), it is also shown in \cite{DGS12} that this hyperbolic mass $m$ is bounded below by the integral quantity

\begin{equation*}
    \frac{1}{2(n-1)w_{n-1}}\int_{\Sigma_{0}} f H d\sigma,
\end{equation*}
where $\Sigma_{0} = \partial \Omega_{0}$ for the domain $\Omega_{0}$ in item (B) is considered as a hypersurface in $\mathbb{H}^{n}$. This is the left-hand side of \eqref{minkowski_area}, and so de Lima and Girao concluded from this inequality that

\begin{equation} \label{penrose}
    m \geq \frac{1}{2} \left( \left( \frac{|\Sigma_{0}|}{w_{n-1}} \right)^{\frac{n-2}{n-1}} + \left( \frac{|\Sigma_{0}|}{w_{n-1}} \right)^{\frac{n}{n-1}} \right)
\end{equation}
whenever $\Sigma_{0}^{n-1} \subset \mathbb{H}^{n}$ is star-shaped with $H>0$. Inequality \eqref{penrose} is a \textit{Penrose-type inequality} in general relativity, and it is conjectured to hold for all asymptotically hyperbolic manifolds $(M^{n},g)$ with horizon boundary $\partial M$. The idea of considering the Penrose inequality for graphical submanifolds and applying the Minkowski inequality goes back to the work of Lam \cite{L10}.

By extending \eqref{minkowski_area}, we can remove the additional assumptions from \cite{LG12} on the underlying $\Omega_{0}$. Thus, we obtain the Penrose inequality for balanced asymptotically hyperbolic graphs under the geometrically natural assumption that $\Omega_{0}$ is outer-minimizing.

\begin{Corollary}[Penrose Inequality for Balanced Asymptotically Hyperbolic Graphs]
Let $M^{n} \subset (\mathbb{H}^{n+1},\overline{g})$, $3 \leq n \leq 7$, satisfy (A)-(C). Suppose that the domain $\Omega_{0}$ in item (B) is outer-minimizing. Then the Penrose inequality \eqref{penrose} holds for $(M^{n},\overline{g}|_{M^{n}})$. Furthermore, equality holds if and only if $\Sigma_{0}= \mathbb{S}^{n-1} \times \{ r_{0} \}$ and $(M^{n},\overline{g}|_{M^{n}})$ is isometric to anti de Sitter Schwarzschild space of mass $m$.
\end{Corollary}

\begin{figure}
\begin{center}
\textbf{The Penrose Inequality for Balanced Asymptotically Hyperbolic Graphs}
\vspace{0.5cm}
\end{center}
    \centering
    \begin{tikzpicture}[scale=2.5]
    \draw [black, thick, fill=blue, opacity=0.2] plot coordinates {(-3,-1) (-2,0) (3,0) (2,-1) };
    \draw [blue, thick, fill=blue, opacity=0.3] plot [smooth cycle] coordinates {(-0.5,-0.7) (-0.2,-0.8) (0.2,-0.9) (0.5, -0.7) (0.3, -0.5) (0.4, -0.3) (0.3,-0.15) (0.2,-0.1) (0,-0.1) (-0.4,-0.3) (-0.6,-0.15) (-0.8,-0.3) (-0.7,-0.6) (-0.6,-0.65)};
    \draw [thick, ->] (0,-0.5) -- (0,2);
    \draw (0,2) node[anchor=west]{\large{$\tau$}};
    \draw (1.7,0) node[anchor=south]{\color{blue}{$\mathbb{H}^{n} \times \{ \tau = \tau_{0} \}$}};
    \draw (-0.7,1.1) node{\large{$(\mathbb{H}^{n+1},\overline{g})$}};
    \draw (0,-0.5) node[anchor=east]{\color{blue}{\large{$\Omega_{0}$}}};
    \draw [black, thick] plot [smooth] coordinates {(0.5,-0.7) (0.8,0) (1.1,0.3) (1.4,0.5) (1.7,0.6) (2,0.7)};
      \draw [black, thick] plot [smooth] coordinates {(-0.8,-0.3) (-0.9, 0.3) (-1.2,0.6) (-1.5,0.8) (-1.8,0.9)};
      \draw [black, thick, dotted] plot [smooth cycle, very thick] coordinates{(-0.9,0.3) (-0.4,0) (0.1,-0.1) (0.6,0) (0.8,0.1) (1.1,0.3) (0.8,0.4) (0.6,0.6) (0,0.65) (-0.5,0.45)};
      \draw (1.5,0.8) node{$M^{n}=\text{graph}(h)$};
    \end{tikzpicture}
    \caption{The total mass of a balanced asymptotically hyperbolic graph in $\mathbb{H}^{n+1}$ which intersects a $\tau=\text{const}$ slice orthgonally along its boundary is bounded below by the right-hand side of \eqref{minkowski_area}. Given Theorem \ref{minkowski}, the Penrose inequality holds for such a graph whenever $\Omega_{0}$ is outer-minimizing for $3 \leq n \leq 7$.}
\end{figure}
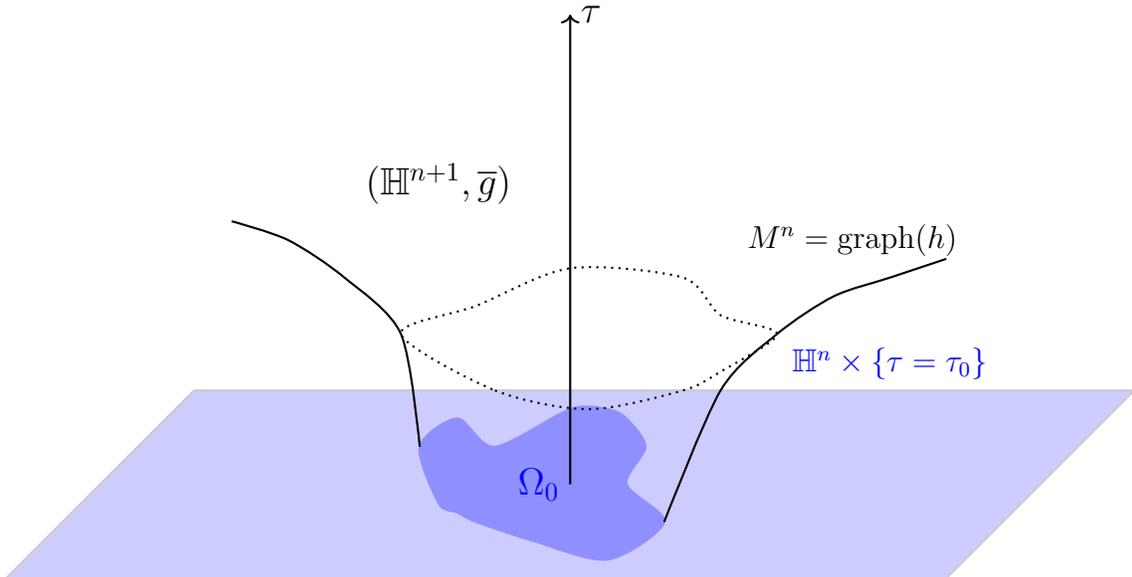
In addition to Minkowski inequalities, the proof method of Theorem \ref{star_shape} also gives a classification of proper weak IMCF on $\mathbb{H}^{n} \setminus \{ 0 \}$ in all dimensions. The blow-down limit of weak IMCF on an asymptotically flat manifold is given by a proper weak IMCF on $\mathbb{R}^{n} \setminus \{ 0 \}$, and a key component of the blow-down lemma in the asymptotically flat setting is that expanding spheres are the only solution of this form (Proposition 7.2 in \cite{HI99}). The gradient estimate \eqref{support_bound} produces the analogue in hyperbolic space.

\begin{Theorem}[Rigidity of IMCF Cores in Hyperbolic Space] \label{rigidity}
The only proper weak solution of IMCF on $\mathbb{H}^{n} \setminus \{ 0 \}$ is the expanding sphere solution
\begin{equation*}
    u(r)= (n-1) \log \left( \frac{\sinh(r)}{\sinh(r_{0})} \right), \hspace{1cm} r_{0} \in \mathbb{R}^{+}.
\end{equation*}  
\end{Theorem}

We remark that the blow-down procedure in \cite{HI99}, which depends in particular on their estimate (3.1), does not carry over to the asymptotically hyperbolic setting-- if it did, Corollary \ref{rigidity} would imply that $\Sigma_{t}$ converge to a round sphere after scaling, and is false according to \cite{N10}, \cite{HW14}. Nevertheless, Corollary \ref{rigidity} has some interesting applications. For instance, it implies a rigidity result for smooth ancient flows in $\mathbb{H}^{n}$-- namely, the expanding sphere solution is the only embedded ancient solution to \eqref{IMCF} converging to a point as $t \rightarrow -\infty$.


We now give an outline of the paper. In Section 2, we review the variational formulation of \eqref{dirichlet}. The aspects of weak solution theory which are especially important to our approach are the comparison principle for weak solutions and the approximation theorems for the level sets.

In Section 3, we prove Theorem \ref{star_shape}. The approach is similar to an Alexandrov reflection method for viscosity solutions of extrinsic curvature flows in $\mathbb{R}^{n}$  that was developed by Chow-Gulliver \cite{CG01}. In \cite{H19}, I applied a version of their reflection method to proper weak IMCF in $\mathbb{R}^{n}$. Here, reflections about a plane in Euclidean space are replaced with inversions about a sphere in the Poincare ball. We derive a comparison principle for weak solutions under these inversions which may be used to show in arbitrary dimensions that $\Sigma_{t}$ are each radial Lipschitz graphs once they escape $B_{r_{+}}(0)$, with a Lipschitz constant depending only on $r_{+}$. Given that $\Sigma_{t}$ are $C^{1,\alpha}$ hypersurfaces when $3 \leq n \leq 7$, we may then Theorem 1.2 from \cite{LW17}, which states that the region outside of a $C^{1}$ star-shaped domain $\Omega_{0} \subset \mathbb{H}^{n}$ with bounded non-negative weak mean curvature is foliated by a star-shaped solution to \eqref{IMCF}. From this, Theorem \ref{star_shape} follows. Theorem \ref{rigidity} also arises from the gradient estimate \eqref{support_bound}.

In Section 5, we consider the Heintze-Karcher type inequality established by Brendle in \cite{B13}, which states that

\begin{equation*}
   (n-1) \int_{\Sigma_{0}} \frac{f}{H} d\sigma \geq n \int_{\Omega_{0}} f d\Omega 
\end{equation*}
for a domain $\Omega_{0} \subset \mathbb{H}^{n}$ with $C^{\infty}$ $H> 0$ boundary. This inequality is key to showing monotonicity of the Minkowski functional under smooth IMCF, and so we must also apply it to obtain monotonicity under weak IMCF. When $n < 8$, the $\Sigma_{t}$ are $C^{1}$ hypersurfaces with bounded non-negative weak mean curvature $H$. Our approach is to approximate $\Sigma_{t}$ by mean curvature flow (MCF), an approach that was utilized in \cite{HI08} and subsequently in \cite{LW17} and \cite{SZ21}. This inequality holds on the approximating surfaces $\Sigma_{\epsilon}$ because they are smooth with $H_{\epsilon} >0$, and so an analysis of $H^{-1}_{\epsilon}$ as $\epsilon \rightarrow 0$ allows us to pass the Heintze-Karcher inequality to $\Sigma_{t}$. An approximation procedure would likely be more complicated for $n \geq 8$, and so we ignore these dimensions here.

In Section 6, we extend inequality \eqref{minkowski} to smooth, outer-minimizing $\Omega_{0} \subset \mathbb{H}^{n}$. To accomplish this, we show monotonicity of the functional

\begin{equation*}
    Q(t) =  |\Sigma_{t}|^{\frac{2-n}{n-1}} \left( \int_{\Sigma_{t}} fH d\sigma - n(n-1) \int_{\Omega_{t}} f d\Omega  \right)
\end{equation*}
under proper weak IMCF in these dimensions. We estimate the growth of the bulk term and the surface integral term separately. First for the bulk term, we use Brendle's inequality for the weak $\Sigma_{t}$ and the co-area formula for Lipschitz functions in order to derive a lower bound on $\int_{\Omega_{t}} f d\Omega$. For the growth of the surface integral term, we employ the method of Freire-Schwartz \cite{FS13} and Wei in \cite{W18} of approximation by elliptically-regularized solutions $u^{\epsilon}$. In particular, we prove an inequality over the level sets of $u^{\epsilon}$ involving the potential $f$ which can then by used to estimate $\int_{\Sigma_{t}} f H d\sigma$. Altogether, we get $Q(0) \geq Q(t)$ for any $t$, and Theorem \ref{star_shape} then completely solves the convergence issue since $Q(t) \geq (n-1) w_{n-1}^{\frac{1}{n-1}}$ for $t > T$.

Finally in Section 7, we extend the pure area Minkowski inequality \eqref{minkowski_area}. Like in \cite{LG12}, it is sufficient to prove monotonicity only for times $t$ satisfying

\begin{equation*}
    \frac{n}{w_{n-1}} \int_{\Omega_{t}} f d\sigma < \left( \frac{|\Sigma_{t}|}{w_{n-1}} \right)^{\frac{n}{n-1}},
\end{equation*}
and this follows immediately from the estimates in Section 5. We will not discuss the RPI for asymptotically hyperbolic graphs in this section-- see \cite{DGS12} for the derivation of the inequality $m \geq \frac{1}{2(n-1)w_{n-1}} \int_{\Sigma_{0}} fH d\sigma$.

\subsection*{Acknowledgements} I would like to thank Yong Wei for explaining his work on IMCF and Minkowski inequalities as well as Martin Li and Raphael Tsiamis for suggestions about the $n>7$ case of Theorem \ref{star_shape}. I would also like to thank the National Center for Theoretical Sciences mathematics division and Copenhagen Center for Geometry and Topology for continued financial and career support.
\section{Preliminaries} 
To introduce weak solutions of IMCF, we consider the functional

\begin{equation} \label{functional}
    J_{u}(F) = |\partial^{*} F| - \int_{F} |\nabla u| d\Omega,
\end{equation}
defined over bounded domains $F \subset \mathbb{H}^{n}$ of finite perimeter $|\partial^{*}F|$, where $u \in C^{0,1}_{loc}(\mathbb{H}^{n})$ is a locally Lipschitz function. Given a proper function$u \in C^{\infty}(\mathbb{H}^{n})$ with $du \neq 0$, one may check that the domains $\Omega_{t} =\{ u < t\}$ minimize $J_{u}$ for each $t >0$ if and only if $u$ solves the Dirichlet problem \eqref{dirichlet} on $\mathbb{H}^{n} \setminus \Omega_{0}$. In turn, the level sets $\Sigma_{t} = \{ u = t\}$ are a solution to classical IMCF \eqref{IMCF}, and this leads to the following definition.

\begin{Definition} \label{weak_imcf}
Let $\Omega_{0} \subset \mathbb{H}^{n}$ be a bounded domain with connected $C^{\infty}$ boundary $\Sigma_{0}$. A \textit{weak solution of IMCF with initial condition $\Omega_{0}$} is a function $u \in C^{0,1}_{\text{loc}} (\mathbb{H}^{n})$ such that $\Omega_{0}=\{ u < 0 \}$ and for each $t >0$ the sub-level sets $\Omega_{t}= \{ u < t \}$ satisfy

\begin{equation} \label{variational}
 J_{u}(\Omega_{t}) \leq J_{u}(F) \hspace{1cm} \forall F \supset \Omega_{0}.
\end{equation}
\end{Definition}

This formulation is inspired by the level-set formulation for mean curvature flow due to Evans-Spruck \cite{ES91} and Chen-Goto-Giga \cite{CGG91}. Because $\Omega_{t}$ solve a minimization problem, the weak flow surfaces $\Sigma_{t}=\partial \Omega_{t}$ enjoy several of the same regularity properties as area-minimizers in Riemannian manifolds. Before stating these, we remind the reader that the \textit{weak mean curvature} of a $C^{1}$ hypersurface $\Sigma_{0}$ is the unique measurable function $H$ on $\Sigma_{0}$ satisfying

\begin{equation}
    \frac{d}{d\epsilon}\bigg\rvert_{\epsilon =0} |X_{\epsilon}(\Sigma_{0})| = \int_{\Sigma_{0}} H V^{\perp} d\sigma
\end{equation}
for the flow $X: \mathbb{H}^{n} \times (-\epsilon_{0},\epsilon_{0}) \rightarrow \mathbb{H}^{n}$ of a vector field $V \in C^{1}_{c}(\Gamma \mathbb{H}^{n})$. 
\begin{Theorem}[\cite{HI99}, Elementary Properties of Proper Weak IMCF] \label{imcf_properties}
Let $\Omega_{0} \subset \mathbb{H}^{n}$ be a bounded domain with connected $C^{\infty}$ boundary $\partial \Omega_{0}=\Sigma_{0}$, and let  $u \in C^{0,1}_{\text{loc}}(\mathbb{H}^{n})$ be a proper weak IMCF with initial condition $\Omega_{0}$. Then the following properties hold with references from \cite{HI99} provided:

\begin{enumerate}[label=\roman*]
    \item (Theorem 1.3) If $3 \leq n \leq 7$, then $\Sigma_{t}= \partial \{ u < t \}$ and $\Sigma^{+}_{t} = \partial \{ u > t \}$ are $C^{1,\alpha}$ hypersurfaces for every $t > 0$.
    \item ((1.10)) If $n <8$, we have the convergence

    \begin{equation}
        \Sigma_{s} \rightarrow \Sigma_{t} \hspace{1cm} \text{  as  }  \hspace{0.5cm} s \nearrow t, \hspace{1cm} \Sigma_{s} \rightarrow \Sigma^{+}_{t} \hspace{1cm} \text{   as   } s \searrow t,
    \end{equation}
    locally in $C^{1,\alpha}$ for every $t > 0$.
    \item ((1.12) and Lemma 5.1) The weak mean curvature $H$ of $\Sigma_{t}$ is essentially bounded on $\Sigma_{t}$. Furthermore, for a.e. $t \in (0,\infty)$ and $\mathcal{H}^{n-1}$ a.e. $x \in \Sigma_{t}$ we have $H(x) = |\nabla u(x)| > 0$.
    \item (Property 1.4 (iii)) For each $t \geq 0$, $\Omega_{t}^{+}$ is the strictly-minimizing hull of $\Omega_{t}$.
    \item (Lemma 5.6) If $\Omega_{0}$ is outer-minimizing, then the perimeter of $\Omega_{t}$ is given by $|\Sigma_{t}|=e^{t}|\Sigma_{0}|$.
\end{enumerate}
When $n \geq 8$, the first two items are true away from a singular set $Z$ with Hausdorff dimension at most $n-8$ and disjoint from $\Omega_{0}$.
\end{Theorem}
In view of the functional $J_{u}$ in \eqref{functional}, the domains $\Omega_{t} = \{ u < t \}$ and $\Omega^{+}_{t}= \{ u \leq t \}$ are respectively outer-minimizing and strictly outer-minimizing domains of $\mathbb{H}^{n}$ for each $t > 0$. If $\Omega_{t} \neq \Omega^{+}_{t}$ for some $t$, then the weak flow $\Sigma_{t}$ jumps across the region $\{ u = t \} \subset M^{n}$ and continues from $\Sigma^{+}_{t}$ according to property (ii). It is important to note, however, that these jumps can only occur at a measure zero subset of times.

In general, weak solutions of IMCF are not unique: for example, if $u$ satisfies \eqref{variational}, then so does $u^{t}= \min\{ u, t \}$ for any choice of $t$, a fact which we will also use later. However, a comparison principle from \cite{HI99} guarantees that there exists a unique weak solution for compact level sets, i.e. proper, for any choice of smooth initial condition $\Omega_{0}$. We will in fact need their entire comparison theorem in order to prove Theorem \ref{star_shape}, so we will re-state it here.
\begin{Theorem}[\cite{HI99}, Theorem 2.2] \label{HI_comp}
Let $u$ and $v$ be weak solutions of IMCF on $\mathbb{H}^{n}$ with initial conditions $\Omega_{0}$ and $\widetilde{\Omega}_{0}$, respectively.

\begin{enumerate}[label=\roman*]
    \item If $ \{ v > u \} \subset \subset W$ on an open set $W \subset \mathbb{H}^{n} \setminus \left(\Omega_{0} \cup \widetilde{\Omega_{0}}\right)$, then $ v \leq u $ on $W$. 
    \item If $\Omega_{0} \subset \widetilde{\Omega}_{0}$ and $u$ is proper, then $v \leq u$ on $\mathbb{H}^{n} \setminus \Omega_{0}$, or equivalently $\Omega_{t} \subset \widetilde{\Omega}_{t}$ for every $t >0$.
    \item Given any $C^{\infty}$ domain $\Omega_{0} \subset \mathbb{H}^{n}$, the proper weak solution $u \in C^{0,1}_{\text{loc}}(\mathbb{H}^{n})$ to IMCF with initial condition $\Omega_{0}$ is unique.
\end{enumerate}
\end{Theorem}

 To conclude, we adress the approximating solutions of a proper weak IMCF, since we will require these in Section 5. In order to prove existence of proper weak solutions to \eqref{variational}, Huisken and Ilmanen in \cite{HI99} considered the elliptic Dirichlet problem

\begin{eqnarray}
    \text{div} \left( \frac{\nabla u^{\epsilon}}{\sqrt{|\nabla u^{\epsilon}|^{2} + \epsilon^{2}}} \right) &=& \sqrt{|\nabla u^{\epsilon}|^{2} + \epsilon^{2}} \hspace{1cm} \text{on} \hspace{1cm} F_{L}, \nonumber \\
    u^{\epsilon} &\equiv& 0 \hspace{3.25cm} \text{on} \hspace{1cm} \Sigma_{0}=\partial \Omega_{0}, \label{u_epsilon} \\
    u^{\epsilon}_{\partial F_{L}} &=& L \hspace{3.25cm} \text{on} \hspace{1cm} \partial F_{L} \setminus \Sigma_{0}. \nonumber
\end{eqnarray}
Here, $F_{L}= \{ v < L \}$ is defined using a subsolution at infinity of \eqref{variational}, but in $\mathbb{H}^{n}$ we can simply pick $v$ to be the expanding sphere solution. We will not summarize the existence argument here (See Section 2 of \cite{W18} for a summary), but these $u^{\epsilon}$ are useful to consider for geometric applications because they are smooth and proper. As in \cite{W18}, they will be a key ingredient in the monotonicity proof, and so we mention the important convergence properties of these.

Before stating these properties, we highlight that because $u^{\epsilon}$ are proper and smooth, the level set $\Sigma^{\epsilon}_{t}= \{ u^{\epsilon} = t \}$ is regular for a.e. $t \in (0,\infty)$ by the $C^{1}$ Sard Theorem (see, for example, Section 2.7 of \cite{S18}). Moreover, the Hessian of $u^{\epsilon}$ solving \eqref{u_epsilon} is non-vanishing at critical points, and so a.e. $x \in \Omega^{\epsilon}_{t}= \{ u^{\epsilon} < t \}$ belongs to a regular level set.
\begin{Theorem}[\cite{W18} Secs 2 and 4, Convergence of Approximate Solutions] \label{approximate}
Let $\Omega_{0} \subset \mathbb{H}^{n}$ be a bounded domain with connected $C^{\infty}$ boundary $\Sigma_{0}$, and let $u^{\epsilon}$ solve \eqref{u_epsilon} on $F_{L}$. Then there is a subsequence $\epsilon_{i} \rightarrow 0$ so that for a.e. $t \geq 0$, the mean curvature $H_{\Sigma^{i}_{t}}= H^{i}$ of $\Sigma^{i}_{t}= \partial \Omega^{i}_{t}$, $\Omega^{i}_{t}= \partial \{ u_{\epsilon_{i}} < t\}$, converges to the weak mean curvature $H_{\Sigma_{t}}=H$ of $\Sigma_{t}$ locally in $L^{2}$. Here $\Sigma_{t}$ correspond the proper weak IMCF with initial condition $\Omega_{0}$. Furthermore, for any $\phi \in C^{0}_{c}(\Omega_{t})$ we have the convergences
\begin{eqnarray}
    \int_{\Omega^{i}_{t}} \phi H_{i} d\Omega &\rightarrow& \int_{\Omega_{t}} \phi H d\Omega, \label{bulk_convergence} \\
    \int_{\Omega^{i}_{t}} \phi H_{i}^{2} d\Omega &\rightarrow& \int_{\Omega_{t}} \phi H^{2} d\Omega. \nonumber
 \end{eqnarray}
\end{Theorem}

\section{An Alexandrov Reflection Principle in Hyperbolic Space}
To prove Theorem \ref{star_shape}, we will use the Poincare ball model for hyperbolic space. This consists of the unit ball $B_{1}(0) \subset \mathbb{R}^{n}$ equipped with the metric

\begin{equation*}
    g_{\mathbb{H}^{n}} = \frac{4}{(1-\rho^{2})^{2}} g_{\mathbb{R}^{n}},
\end{equation*}
where $\rho$ is the radial coordinate of $B_{1}(0)$. This model is related to \eqref{hyperbolic_space} by the change of variable
\begin{eqnarray}
    r(\rho) &=& \log \left( \frac{1+ \rho}{1-\rho} \right). \label{transformation1}
\end{eqnarray}
Throughout this section, we will denote the position vector of a point $x= (x_{1},\dots, x_{n}) \in \mathbb{R}^{n}$ by $\mathbf{x}$.

In \cite{H19} and following earlier work by Chow-Gulliver in \cite{CG01}, I studied weak solutions of IMCF in $\mathbb{R}^{n}$ using a parabolic analogue to the moving plane method of Alexandrov \cite{A58} for constant mean curvature hypersurfaces. In $\mathbb{R}^{n}$, composition of a weak solution of IMCF with reflection about a fixed hyperplane produces a new weak solution. The key in \cite{H19} is a comparison principle within the lower half-plane between the original solution and the reflected one. For a survey of the Euclidean version of this method, see \cite{C23}.

Here we consider sphere inversions in the Poincare ball, which are the analogues of reflections. For every constant $\lambda \in (0,1)$ and point $\theta \in \partial B_{1}(0)$, we consider the ball $B_{R_{\lambda}}(x_{\lambda,\theta}) \subset \mathbb{R}^{n}$ with radius 

\begin{equation}
    R_{\lambda} = \frac{1}{2} (\lambda^{-1} - \lambda)
\end{equation}
and centered at the point

\begin{equation}
      \mathbf{x}_{\lambda,\theta} = \frac{1}{2} (\lambda^{-1} + \lambda) \boldsymbol{\theta} \in \mathbb{R}^{n} \setminus B_{1}(0).
\end{equation}
$B_{R_{\lambda}}(x_{\lambda,\theta})$ lies a distance $\lambda$ from $0 \in B_{1}(0)$, is centered on the line containing $0$ and $\theta$, and intersects $\partial B_{1}(0)$ orthogonally. Given these facts, one may verify that inversion $F_{\lambda, \theta}: \mathbb{R}^{n} \setminus \{ x_{\lambda,\theta} \} \rightarrow \mathbb{R}^{n} \setminus \{ x_{\lambda,\theta} \}$ about $B_{R_{\lambda,\theta}}(x_{\lambda,\theta})$ defined by

\begin{equation} \label{inversion}
    \mathbf{F}_{\lambda, \theta} (\mathbf{x}) = \frac{R_{\lambda}^{2}}{||\mathbf{x} - \mathbf{x_{\lambda,\theta}}||^{2}} \left( \mathbf{x} - \mathbf{x}_{\lambda,\theta} \right) + \mathbf{x}_{\lambda,\theta}
\end{equation}
restricts to an isometry over $(B_{1}(0),g_{\mathbb{H}^{n}})$. For a domain $\Omega_{0} \subset B_{1}(0)$ we denote

\begin{equation}
  \Omega^{\lambda,\theta}_{0} = F_{\lambda,\theta}(\Omega_{0}),  
\end{equation}
and consider the ``lower half-space" region $H_{\lambda,\theta} \subset B_{1}(0)$ given by

\begin{equation} \label{half_space}
    H_{\lambda,\theta} = B_{1} (0) \cap B_{R_{\lambda}} (x_{\lambda,\theta}).
\end{equation}
Given a proper weak solution $u: (B_{1}(0),g_{\mathbb{H}^{n}}) \rightarrow \mathbb{R}$ of IMCF, we define another weak solution  $u_{\lambda,\theta}: (B_{1}(0),g_{\mathbb{H}}^{n+1}) \rightarrow \mathbb{R}$ by

\begin{equation}
    u_{\lambda,\theta}(x) = u \circ F_{\lambda,\theta}(x).
\end{equation}    
First, we derive a comparison principle between $u$ and $u_{\lambda,\theta}$. Geometrically, the following proposition gives that if $\Omega_{0}^{\lambda,\theta} \setminus \overline{H_{\lambda,\theta}} \subset \Omega_{0}$, then the inverted image $\Omega_{t}^{\lambda,\theta} \setminus \overline{H_{\lambda,\theta}}$ gets trapped in the evolving domain $\Omega_{t}$ for every $t >0$, see Figure \ref{comp1}.
\begin{Proposition}[Comparison Principle for Inverted Weak Solutions] \label{comp}
Let $u: (B_{1}(0),g_{\mathbb{H}^{n}}) \rightarrow \mathbb{R}$ be a variational solution to IMCF with initial condition $\Omega_{0}$. Fix $\theta \in \partial B_{1}(0)$ and $\lambda \in (0,1)$. If $ \Omega^{\lambda,\theta}_{0} \setminus \overline{H_{\lambda,\theta}} \subset \Omega_{0}$, then $ \Omega_{t}^{\lambda,\theta} \setminus \overline{H_{\lambda,\theta}} \subset \Omega_{t}$ for every $t >0$. Equivalently, $u_{\lambda,\theta}(x) \leq u(x)$ for every $x \in H_{\lambda,\theta} \setminus \overline{\Omega_{0}}$.
\end{Proposition}

\begin{proof}
For any $t > 0$, the function $u^{t}(x) = \min\{ t, u(x) \}$ is a variational solution to IMCF on $\mathbb{H}^{n} \setminus \Omega_{0}$. Define the set

\begin{equation*}
    W= \Omega^{\lambda,\theta}_{t} \setminus \overline{ \left(\Omega_{0} \cup H_{\lambda,\theta} \right) }. 
\end{equation*}
For $\delta > 0$, we compare $u^{t}$ and $u_{\lambda,\theta} + \delta$ on $W$. By definition, $u^{t} < u_{\lambda,\theta} + \delta= t + \delta$ on $\partial \Omega^{\lambda,\theta}_{t}$, and by assumption $u_{\lambda,\theta}+ \delta > u^{t} = 0$ on $\partial \Omega_{0} \setminus H_{\lambda,\theta}$. Finally, since $F_{\lambda,\theta}$ acts as the identity on $\partial B_{\lambda}(x_{\lambda,\theta}) \cap B_{1}(0)$, we have that $u_{\lambda,\theta}= u$ at these points and so altogether

\begin{equation*}
    u^{t} < u_{\lambda,\theta} + \delta \hspace{2cm} \text{  on  } \hspace{1cm} \partial W \subset  \partial B_{\lambda} (x_{\lambda},\theta) \cup \partial \Omega^{\lambda,\theta}_{t} \cup \left( \partial \Omega_{0} \setminus H_{\lambda,\theta} \right).
\end{equation*}
By Theorem \ref{HI_comp}(i), this implies that $u^{t} \leq u_{\lambda,\theta} + \delta$ on $W$, and since $\delta$ is arbitrary we get $u \leq u_{\lambda,\theta} < t$ in $W$. Hence $\Omega^{\lambda,\theta}_{t} \setminus  \overline{H_{\lambda,\theta}}  \subset \Omega_{t} \setminus \overline{H_{\lambda,\theta}}$. This is equivalent to $u_{\lambda,\theta}(x) \geq u(x)$ for $x \in B_{1}(0) \setminus \overline{\left( H_{\lambda,\theta} \cup \Omega^{\lambda,\theta}_{0} \right)}$, and also to $u(x) \leq u_{\lambda,\theta}(x)$ for $x \in H_{\lambda,\theta} \setminus \Omega_{0}$. 
\end{proof}
\begin{figure}
\begin{center}
\textbf{An Alexandrov Principle in the Poincare Ball}
\vspace{0.5cm}
\end{center}
    \centering
    \begin{tikzpicture}[scale=2.5, rotate=-30]

    
 \draw[line width=2 pt](0,0) circle (1.0) ;
\draw[ultra thick] (0.8,0.6)--(0.7,0.525);
 \draw[ultra thick] (0.7,0.525) -- (0.63,0.61833);
 
 \draw[ultra thick] (-0.8,0.6) -- (-0.7,0.525);
  \draw[ultra thick] (-0.7,0.525) -- (-0.63,0.61833);
 
  \draw[line width= 2 pt] (0,1.45) circle (1.05);
  \draw (0,1.7) node{$B_{R}(x_{\lambda,\theta})$};
  \draw (0,0.75) node{\color{blue}{$H_{\lambda,\theta}$}};
  \draw (0.3,-0.6) node{$\left( B_{1}(0),g_{\mathbb{H}^{n}} \right)$};
  \draw (0,0) node{\Large{.}};
  \draw (0.2,0) node{$\Omega_{0}$};
  \draw (0,0) -- (0,0.4);
 \draw (0,0.2) node[anchor=west]{$\lambda$};
 \draw (0,1) node{\Huge{.}};
 \draw (0,1) node[anchor=south] {$\theta$};
 \draw (0,1.45) node{\Huge{.}};
 \draw (0,1.45) node[anchor=north]{$\mathbf{x_{\lambda,\theta}}$};
 \draw [black, line width=2.5 pt, fill=gray, opacity=0.3] plot [smooth cycle] coordinates {(0,0.5) (-0.1,0.35) (-0.2,0.35) (-0.4,0.1) (-0.1,-0.35)  (0.1,-0.2) (0.2,-0.3) (0.4,-0.2) (0.4,0.1) (0.3,0.2) (0.1,0.35)};

  \begin{scope}
   \clip (0,1.45) circle[radius=1.05];
   \fill[blue!50, opacity=0.2] (0,0) circle[radius=1]; 
  \end{scope}
\begin{scope}
    \clip  plot coordinates{(-0.07,0.41) (-0.06,0.33) (-0.04,0.28) (-0.02,0.25) (0,0.2) (0.02,0.25) (0.04,0.28) (0.06,0.33) (0.07,0.41) }; 
   \fill[color=blue, fill=blue, line width=2 pt, opacity=0.2] plot coordinates{(0,0.5) (-0.1,0.35) (-0.2,0.4) (-0.4,0.1) (-0.1,-0.35)  (0.1,-0.2) (0.2,-0.3) (0.4,-0.2) (0.4,0.1) (0.3,0.2) (0.1,0.35)};
\end{scope}
\draw (1.8,1) node{\Huge{$\Rightarrow$}};

  \begin{scope}[shift={(3.5,1.95)}]
 \draw[line width=2 pt](0,0) circle (1.0) ;
\draw[ultra thick] (0.8,0.6)--(0.7,0.525);
 \draw[ultra thick] (0.7,0.525) -- (0.63,0.61833);
 
 \draw[ultra thick] (-0.8,0.6) -- (-0.7,0.525);
  \draw[ultra thick] (-0.7,0.525) -- (-0.63,0.61833);
  
  \draw (0,0.75) node{\color{blue}{$H_{\lambda,\theta}$}};
 \draw[thick](0,0) circle (1.0) ;
  \draw[line width=2 pt] (0,1.45) circle (1.05);
  \draw (0,1.7) node{$B_{R}(x_{\lambda,\theta})$};
  \draw (0.3,-0.6) node{$\left( B_{1}(0),g_{\mathbb{H}^{n}} \right)$};
  \draw (0,0) node{\Large{.}};
  \draw (0.2,0) node{$\Omega_{t}$};
  \draw (0,0) -- (0,0.4);
 \draw (0,1) node{\Huge{.}};
 \draw (0,1) node[anchor=south] {$\theta$};
 \draw (0,1.45) node{\Huge{.}};
 \draw (0,1.45) node[anchor=north]{$\mathbf{x_{\lambda,\theta}}$};
 \draw [black, line width=2.5 pt, fill=gray, opacity=0.3] plot [smooth cycle] coordinates {(0,0.6) (-0.1,0.55) (-0.3,0.5) (-0.6,0.2) (-0.3,-0.4)  (0,-0.5) (0.2,-0.35) (0.45,-0.25) (0.4,0.25) (0.3,0.35) (0.1,0.5)};

 \begin{scope}
    \clip plot coordinates{(-0.35,0.46005) (-0.3,0.3) (-0.15,0.2) (-0.1,0.13) (0,0.03) (0.02,0.06) (0.09,0.14) (0.16,0.24) (0.2,0.41922) (0.1,0.404772) (0,0.4) (-0.1,0.404772) (-0.2,0.41922) };
   \fill[blue, ultra thick, fill=blue, opacity=0.2] plot coordinates {(0,0.6) (-0.1,0.55) (-0.3,0.5) (-0.6,0.2) (-0.3,-0.4)  (0,-0.5) (0.2,-0.35) (0.45,-0.25) (0.4,0.25) (0.3,0.35) (0.1,0.5)};
   \end{scope}

  \begin{scope}
   \clip (0,1.45) circle[radius=1.05];
   \fill[blue!50, opacity=0.2] (0,0) circle[radius=1]; 
  \end{scope}
  \end{scope}
    \end{tikzpicture}
    \caption{If the part of $\Omega_{0}$ that lies inside the half-space $H_{\lambda,\theta}$ reflects into the part of $\Omega_{0}$ that lies outside, then this remains true for each $\Omega_{t}$.}
    \label{comp1}
\end{figure}
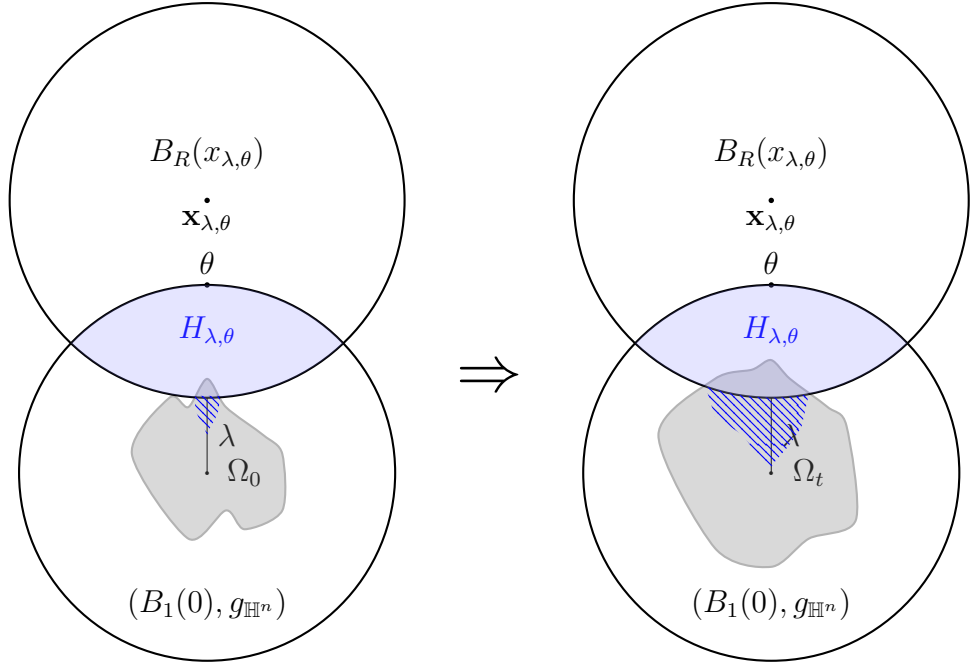

Consider points $x_{1},x_{2} \in B_{1}(0)$ with $\mathbf{x_{1}}= \rho_{1} \boldsymbol{\theta}, \mathbf{x_{2}}= \rho_{2} \boldsymbol{\theta}$ for the same point $\theta \in \partial B_{1}(0)$ and for $\sup_{\Omega_{0}} \rho \leq \rho_{1} < \rho_{2}$. $x_{1}$ and $x_{2}$ may be bisected by a sphere $\partial B_{R_{\lambda}} (x_{\lambda,\theta})$ so that $u(x_{2})= u_{\lambda,\theta}(x_{1})$ and $H_{\lambda,\theta} \cap \Omega_{0} = \varnothing$. By Proposition \ref{comp}, this implies that $u(x_{2}) \geq u(x_{1})$. In fact, we can compare \textit{any} pair of points which are bisected by a sphere that is disjoint from $\Omega_{0}$.

\begin{Theorem} \label{outside}
Let $u: (B_{1}(0),g_{\mathbb{H}^{n}}) \rightarrow \mathbb{R}$ be a variational solution to IMCF with initial condition $\Omega_{0}$, and let $B_{\rho_{+}}(0) \subset B_{1}(0)$ be the smallest ball about $0$ containing $\Omega_{0}$. Consider points $x_{1}, x_{2} \in B_{1}(0)$ with

\begin{eqnarray*}
\mathbf{x}_{1} &=& \rho_{1} \boldsymbol{\theta}_{1}, \\
\mathbf{x}_{2} &=& \rho_{2} \boldsymbol{\theta}_{2},
\end{eqnarray*}
for $\theta_{1}, \theta_{2} \in \partial B_{1}(0)$ and

\begin{equation*}
\rho_{+} < \rho_{1} < \rho_{2} < 1.    
\end{equation*}
Whenever
\begin{equation} \label{lipschitz_est}
   \rho_{2} - \rho_{1} \geq \left(\frac{\rho_{1}\rho_{2} (\rho_{1}^{-1} - \rho_{1})^{2}}{ (\rho_{+}^{-1} - \rho_{+})^{2} - (\rho_{1}^{-1} - \rho_{1})^{2}}\right)^{\frac{1}{2}} ||\boldsymbol{\theta}_{2} - \boldsymbol{\theta}_{1}||,
\end{equation}
we have that $u(x_{2}) \geq u(x_{1})$. 
\end{Theorem}

\begin{proof}
Consider the point $x_{0} \in \mathbb{R}^{n}$ given by

\begin{equation} \label{x'}
    \mathbf{x_{0}}= (1-s_{0})\mathbf{x}_{1} + s_{0} \mathbf{x}_{2}
\end{equation}
for

\begin{equation} \label{s_{0}}
    s_{0} = \frac{1-\rho_{1}^{2}}{\rho_{2}^{2} - \rho_{1}^{2}} = 1 + \frac{1-\rho_{2}^{2}}{\rho_{2}^{2}-\rho_{1}^{2}} > 1.
\end{equation}
Observe first that 

\begin{eqnarray}
    ||\mathbf{x_{0}}||^{2} &=& s_{0}^{2}\rho_{2}^{2} + (1-s_{0})^{2}\rho_{1}^{2} + 2 s_{0} (1-s_{0}) \rho_{1} \rho_{2} \langle \boldsymbol{\theta}_{1}, \boldsymbol{\theta}_{2} \rangle \nonumber \\
    &\geq& s_{0}^{2}\rho_{2}^{2} + (1-s_{0})^{2}\rho_{1}^{2} + 2 s_{0} (1-s_{0}) \rho_{1} \rho_{2}=\left(s_{0}\rho_{2} + (1-s_{0})\rho_{1} \right)^{2}  \label{x'_length} \\
    &=&  \left( \frac{1+ \rho_{1}\rho_{2}}{\rho_{1} + \rho_{2}}\right)^{2} > 1 \nonumber.
\end{eqnarray}
Here, we have used the fact that $(1-s_{0}) < 0$ and $|\langle \boldsymbol{\theta}_{1}, \boldsymbol{\theta}_{2} \rangle| \leq 1$ in the first inequality. Next, we consider the ball $B_{R} (x_{0}) \subset \mathbb{R}^{n}$ for

\begin{equation} \label{R}
    R= ||\mathbf{x_{0}} - \mathbf{x}_{1}||^{\frac{1}{2}} ||\mathbf{x_{0}} - \mathbf{x}_{2}||^{\frac{1}{2}}
\end{equation}
Notice for this choice of $R$ that $x_{2} \in B_{R}(x_{0})$, and so $B_{R}(x_{0}) \cap B_{1}(0) \neq \varnothing$. In terms of radial coordinates, $R$ is given by

\begin{eqnarray}
R^{2} &=& (s_{0}-1)s_{0} ||\mathbf{x}_{2} - \mathbf{x}_{1}||^{2} = (s_{0} -1)s_{0} \left( \rho_{2}^{2} + \rho_{1}^{2} - 2 \rho_{1}\rho_{2} \langle \boldsymbol{\theta}_{1},\boldsymbol{\theta}_{2} \rangle \right) \nonumber \\
&=& (s_{0}-1)s_{0} \left( \rho_{2}^{2} + \rho_{1}^{2} - 2\rho_{1}\rho_{2} + \rho_{1}\rho_{2} ||\boldsymbol{\theta}_{1} - \boldsymbol{\theta}_{2}||^{2} \right) \label{R_length} \\
&=& (s_{0}-1)s_{0} \left(\rho_{2}-\rho_{1} \right)^{2} \left(1+ \rho_{1}\rho_{2} \frac{||\boldsymbol{\theta}_{2} - \boldsymbol{\theta}_{1}||^{2}}{(\rho_{2} -\rho_{1})^{2}} \right) \nonumber \\
&=& \frac{(1-\rho_{1}^{2})(1-\rho_{2}^{2})}{(\rho_{1} + \rho_{2})^{2}} \left(1+ \rho_{1}\rho_{2} \frac{||\boldsymbol{\theta}_{2} - \boldsymbol{\theta}_{1}||^{2}}{(\rho_{2} -\rho_{1})^{2}} \right) \nonumber
\end{eqnarray}
We consider points $\theta \in \partial B_{1}(0) \cap \partial B_{R}(x_{0})$. For these $\theta$, we have

\begin{equation} \label{intersection}
    R^{2} = \langle \boldsymbol{\theta} - \mathbf{x_{0}}, \boldsymbol{\theta} - \mathbf{x_{0}} \rangle = 1 + ||\mathbf{x_{0}}||^{2} -2 \langle \boldsymbol{\theta}, \mathbf{x_{0}} \rangle.
\end{equation}
Using the first line of \eqref{x'_length}, the first line of \eqref{R_length}, and the identity \eqref{s_{0}} for $s_{0}$, we have that

\begin{equation*}
    R^{2} - ||\mathbf{x_{0}}||^{2} = -s_{0} \rho_{2}^{2} + (s_{0} -1) \rho_{1}^{2} = -1,
\end{equation*}
and so from \eqref{intersection}

\begin{equation*}
    \langle \mathbf{x_{0}}, \boldsymbol{\theta} \rangle = 1 \hspace{1cm} \text{for} \hspace{1cm} \theta \in \partial B_{1} (0) \cap \partial B_{R}(x).
\end{equation*}
Thus $\langle \boldsymbol{\theta} - \mathbf{x_{0}}, \boldsymbol{\theta} \rangle = 0$, and so $\partial B_{R}(x_{0})$ intersects $\partial B_{1}(0)$ orthogonally. As a consequence, inversion about $B_{R}(x_{0})$ defines an isometry of the Poincare Ball. The distance $\lambda$ from $B_{R}(x_{0})$ and $0$ is related to $R$ by

\begin{equation*}
    \frac{1}{2} \left( \lambda^{-1} - \lambda \right) = R.
\end{equation*}
Using \eqref{R_length}, we compute an upper bound on $R$:

\begin{eqnarray*}
   R^{2} &=& \frac{(1-\rho_{1}^{2})(1-\rho_{2}^{2})}{(\rho_{1} + \rho_{2})^{2}} \left(1+ \rho_{1}\rho_{2} \frac{||\boldsymbol{\theta}_{2} - \boldsymbol{\theta}_{1}||^{2}}{(\rho_{2} -\rho_{1})^{2}} \right) \nonumber \\
   &\leq& \frac{(\rho_{1}^{-1}-\rho_{1})^{2}}{4} \left( 1 + \frac{(\rho_{+}^{-1} - \rho_{+})^{2} - (\rho_{1}^{-1} - \rho_{1})^{2}}{(\rho_{1}^{-1} - \rho_{1})^{2}} \right) \\
   &=& \frac{1}{4} \left( \rho_{+}^{-1} - \rho_{+} \right)^{2}.
\end{eqnarray*}
For the inequality, we used the fact that $\rho_{1} < \rho_{2}$ and assumption \eqref{lipschitz_est}. Since $\lambda^{-1} - \lambda \leq \rho_{+}^{-1} - \rho_{+}$, we deduce that $\lambda \geq \rho_{+}$. Let $\theta \in \partial B_{1}(0)$ lie on the line segment between $x_{0}$ and $0$. Since $B_{\rho_{+}}(0)$ contains $\Omega_{0}$, we have that $H_{\lambda,\theta} \cap \Omega_{0} = \varnothing$, and so we may apply the previous proposition in $H_{\lambda,\theta}$. In particular,

\begin{equation*}
    u_{\lambda,\theta} (x) \leq u (x), \hspace{2cm} x \in H_{\lambda,\theta}.
\end{equation*}
Since,

\begin{equation*}
    F_{\lambda,\theta} (\mathbf{x}_{2}) = \frac{R^{2}}{||\mathbf{x}_{2} - \mathbf{x}_{0}||^{2}} (\mathbf{x}_{2} - \mathbf{x}_{0}) + \mathbf{x}_{0} = \frac{R^{2}}{||\mathbf{x}_{2}-\mathbf{x}_{0}||^{2}} \frac{||\mathbf{x}_{2}-\mathbf{x}_{0}||}{||\mathbf{x}_{1} - \mathbf{x}_{0}||}(\mathbf{x}_{1} -\mathbf{x}_{0}) + \mathbf{x}_{0}=\mathbf{x}_{1},
\end{equation*}
we have that
\begin{equation*}
    u(x_{2}) \geq u(x_{1}),
\end{equation*}
as claimed.




\end{proof}
\begin{Remark}
Determining the correct relation between the $\rho$ and $\theta$ coordinates of $x_{1}$ and $x_{2}$ to allow for the reflection property was quite tedious, but worth it.
\end{Remark}

\begin{figure}
\begin{center}
\textbf{Comparison outside $B_{\rho_{+}}(0)$}
\vspace{0.5cm}
\end{center}
    \centering

    \begin{tikzpicture}[scale=3]

\tikzset{
    show curve controls/.style={
        decoration={
            show path construction,
            curveto code={
                \draw [blue, dashed]
                    (\tikzinputsegmentfirst) -- (\tikzinputsegmentsupporta)
                    node [at end, cross out, draw, solid, red, inner sep=2pt]{};
                \draw [blue, dashed]
                    (\tikzinputsegmentsupportb) -- (\tikzinputsegmentlast)
                    node [at start, cross out, draw, solid, red, inner sep=2pt]{};
            }
        }, decorate
    }
}

\draw [black, line width=2.5 pt, fill=gray, opacity=0.3] plot [smooth cycle] coordinates {(0,0.4) (-0.1,0.2) (-0.3,0.1) (-0.3,-0.1) (0.1,-0.35)  (0.3,-0.2) (0.35,0) (0.15,0.1) (0.1,0.3)};
    
 \draw[line width= 2 pt](0,0) circle (1.0) ;
 \coordinate (0) at (0,0) ;
 \coordinate (x) at (-0.769232,0.87077);
 \coordinate (y) at (-0.37737,0.42719);
\draw[thick] (0) circle [radius=0.41];
    \draw (0,0) node{\Large{.}};
    \draw (-0.1,1.2) node[anchor=west]{$F_{\lambda,\theta}(x_{1})=x_{2}$};
  \draw (0,0) node[anchor=west]{$\Omega_{0}$};
  \draw (0,0.5) node{\Large{.}};
  \draw (0,0.5) node[anchor=south]{$\mathbf{x_{1}}$};
  \draw (-0.4,0.6928) node{\Large{.}};
  \draw (-0.4,0.6928) node[anchor=south]{$\mathbf{x_{2}}$}; 
  \draw (-0.769232,0.87077) node{\Large{.}}; 
  \draw[line width= 2 pt] (x) circle[radius=0.591617];
  \draw (-0.37514,0.424657) node[anchor=west]{$\lambda$};
  \draw (-0.769232,0.87077) node[anchor=south]{$\mathbf{x}_{0}$};
  \draw (-0.6,0.5) node{\color{blue}{$H_{\lambda,\theta}$}};
  \draw (0,0) -- (-0.37737,0.42719);
  \draw[dotted,thick] (0,0) -- (-0.28284,-0.28284);
  \draw (-0.14142,-0.14142) node[anchor=west]{$\rho_{+}$};
  \draw (-0.9,1.2) node{$B_{R}(x_{0})$};
  \draw[ultra thick] (-0.19,0.88)-- (-0.09,0.88);
  \draw[ultra thick] (-0.09,0.88) -- (-0.09,0.98);
  \draw[ultra thick] (-0.86,0.29) -- (-0.88,0.19);
  \draw[ultra thick] (-0.88,0.19) -- (-0.98, 0.21);
  \begin{scope}
   \clip (x) circle[radius=0.591617];
   \fill[blue!50, opacity=0.2] (0) circle[radius=1]; 
  \end{scope}
  \draw (0,-0.7) node {$\left(B_{1}(0),g_{\mathbb{H}^{n}}\right)$};
    \end{tikzpicture}
    \caption{One can compare $x_{1}$ and $x_{2}$ if their bisecting sphere does not contact $\Omega_{0}$. This happens whenever $x_{1}$ and $x_{2}$ lie outside $B_{\rho_{+}}(0)$ and satisfy \eqref{lipschitz_est}.}
    \label{comp2}
\end{figure}
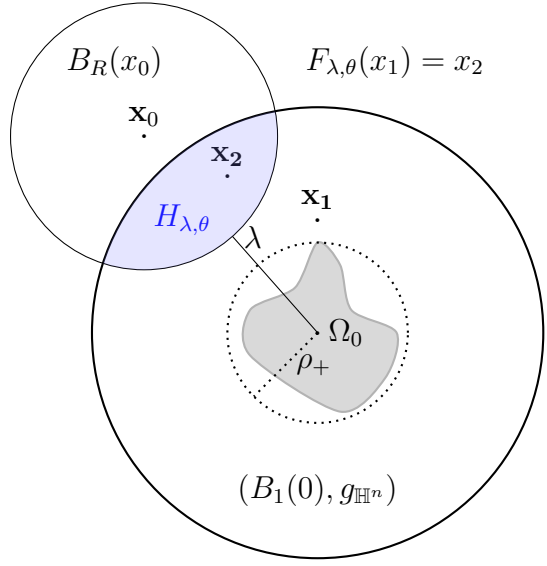

\begin{Corollary}
For each $\mathbf{x}_{1}= \rho_{1} \boldsymbol{\theta}_{1}, \mathbf{x}_{2}= \rho_{2} \boldsymbol{\theta}_{2} \in \Sigma_{t} \setminus B_{\rho_{+}}(0)$, we have that

\begin{equation} \label{lipschitz_bound}
    \frac{\rho_{2} - \rho_{1}}{||\boldsymbol{\theta}_{2} - \boldsymbol{\theta}_{1}||} \leq \left( \frac{\rho_{1}\rho_{2}}{ (\rho_{+}^{-1} - \rho_{+})^{2} - (\rho_{1}^{-1} - \rho_{1})^{2}}\right)^{\frac{1}{2}} (\rho_{1}^{-1} - \rho_{1}).
\end{equation}
As a result, $\Sigma_{t}= \{ (\rho_{t}(\theta), \theta) | \theta \in \partial B_{1}(0) \} \subset \mathbb{H}^{n}$ is a Lipschitz graph satisfying \eqref{lipschitz_bound} for each $t \in (T,\infty)$, where is $T$ given in \eqref{waiting_time}.
\end{Corollary}
\begin{proof}

We proceed by contradiction: suppose \eqref{lipschitz_bound} is false for some $x_{1}, x_{2} \in \Sigma_{t} \setminus B_{\rho_{+}}(0)$, and choose a sequence of points $x^{(i)}_{2} \in \Omega_{t}$ converging to $x_{2}$. If \eqref{lipschitz_bound} is false, then for $i$ large enough we would have $\rho_{2}^{(i)} > \rho_{1}$ and

\begin{equation*}
    \frac{\rho^{(i)}_{2} - \rho_{1}}{||\boldsymbol{\theta}_{2}^{(i)} - \boldsymbol{\theta}_{1}||} \geq \left( \frac{\rho_{1}\rho^{(i)}_{2}}{ (\rho_{+}^{-1} - \rho_{+})^{2} - (\rho_{1}^{-1} - \rho_{1})^{2}}\right)^{\frac{1}{2}} (\rho_{1}^{-1} - \rho_{1}).
\end{equation*}
Theorem \eqref{outside} then implies that $u(x^{(i)}_{2}) \geq u(x_{1})$. On the other hand, we have $u(x_{2}^{(i)}) < t$ and $u(x_{1}) = t$ by definition. This is a contradiction, and \eqref{lipschitz_bound} holds on $\Sigma_{t} \setminus B_{\rho_{+}}(0)$. 

We also argue that $\Sigma_{t} \subset B_{1}(0) \setminus \overline{B_{\rho_{+}}(0)}$ for $t > T$. Setting $\rho_{-}= \min_{\Sigma_{0}} \rho$, one easily verifies that 

\begin{equation} \label{ball_comp}
    \widetilde{\Omega}_{t} = B_{\rho(t)}(0), \hspace{2cm} \rho(t)^{-1} - \rho(t)= e^{-\frac{t}{n-1}} (\rho^{-1} - \rho_{-}), 
\end{equation}
is an IMCF with initial condition $B_{\rho_{-}}(0)$. Using the relation

\begin{equation*}
    \sinh(r) = \frac{1}{2} (\rho^{-1} - \rho)^{-1},
\end{equation*}
we see from \eqref{ball_comp} that for $t > T= -(n-1) \log \left( \frac{\rho_{+}^{-1} - \rho_{+}}{\rho^{-1}_{-} - \rho_{-}} \right)$,
\begin{equation*}
    \rho(t)^{-1} - \rho(t) < \rho^{-1}_{+} - \rho_{-},
\end{equation*}
and hence $\overline{B_{\rho_{+}}(0)} \subset \widetilde{\Omega_{t}}$ for $t \in (T,\infty)$. Since $\widetilde{\Omega}_{0} \subset \Omega_{0}$, $\widetilde{\Omega}_{t} \subset \Omega_{t}$ for each $t >0$ by Theorem \ref{HI_comp} (ii).



\end{proof}

To prove Theorem \ref{star_shape}, we invoke the work of Li-Wei \cite{LW17}. \cite{LW17} establishes a Harnack-type inequality for star-shaped IMCFs in $\mathbb{H}^{n}$ which follows earlier work on star-shaped solutions, c.f. \cite{G11}, \cite{D11}, \cite{S17}, and \cite{L16}. Roughly speaking, this estimate implies that a $C^{1}$ star-shaped $\Sigma_{0} \subset \mathbb{H}^{n}$ with non-negative weak mean curvature may be taken as initial data for smooth IMCF.
\begin{proof}[Proof of Theorem \ref{star_shape} and Theorem \ref{rigidity}]

 Let $u \in C^{0,1}_{\text{loc}}(\mathbb{H}^{n})$ be a proper weak IMCF with initial condition $\Omega_{0}$. For any $t_{0} > T$, \eqref{lipschitz_bound} implies that $\Sigma_{t_{0}}$ is a strictly star-shaped $C^{1,\alpha}$ hypersurface with weak mean curvature $H \geq 0$. Therefore, by \cite{LW17}, Theorem 1.2, there exists a smooth star-shaped IMCF $\{ \widetilde{\Sigma_{t}} \}_{t_{0} < t < \infty}$ that foliates $\mathbb{H}^{n} \setminus \overline{\Omega_{t_{0}}}$. Thus the function $\widetilde{u}: \mathbb{H}^{n} \rightarrow \mathbb{R}$ defined by 

\begin{equation*}
    \widetilde{u}(x) = \begin{cases}
    u(x) & x \in \overline{\Omega_{t_{0}}} \\
    t & x \in \widetilde{\Sigma_{t}} \subset \mathbb{H}^{n} \setminus \overline{\Omega_{t_{0}}}
    \end{cases}
\end{equation*}
is a proper weak solution to IMCF with initial condition $\Omega_{0}$, c.f. Lemma 2.3 in \cite{HI99}. By Theorem \ref{HI_comp}(iii), we must have that $\widetilde{u}=u$ on $\mathbb{H}^{n}$ and hence $\Sigma_{t}$ is smooth for $t \in (T,\infty)$.

Once again from \eqref{lipschitz_bound}, we have for each $t \in (T,\infty)$ that the graph $\rho_{t}(\theta)$ satisfies

\begin{equation} \label{lipschitz}
    |D\rho_{t}|_{\mathbb{S}^{n-1}} \leq \frac{1 - \rho_{t}^{2}}{\left((\rho_{+}^{-1} - \rho_{+})^{2} - (\rho_{t}^{-1} - \rho_{t})^{2}\right)^{\frac{1}{2}}}.
\end{equation}
By the relations

\begin{eqnarray*}
 r(\rho) &=& \log \left( \frac{1+\rho}{1-\rho} \right), \\
    \sinh(r) &=& \frac{2}{\rho^{-1} - \rho},
\end{eqnarray*}
we get for the geodesic variable $r$ that
\begin{eqnarray*}
    |Dr_{t}|_{\mathbb{S}^{n-1}} &=& \frac{2}{1-\rho_{t}^{2}} |D\rho_{t}|_{\mathbb{S}^{n-1}} \\
             &\leq& \frac{2}{\left( (\rho_{+}^{-1} - \rho_{+})^{2} - (\rho_{t}^{-1} - \rho_{t})^{2} \right)^{\frac{1}{2}}}= \frac{1}{\left( \sinh^{-2}(r_{+}) - \sinh^{-2}(r_{t}) \right)^{\frac{1}{2}}} \\
             &=& \frac{\sinh(r_{+}) \sinh(r_{t})}{\left( \sinh^{2}(r_{t}) - \sinh^{2}(r_{+}) \right)^{\frac{1}{2}}}.
\end{eqnarray*}
To briefly address Theorem \ref{rigidity}, let $\{ \Omega_{t} \}_{t \in (-\infty,\infty)}$ be a proper weak IMCF on $\mathbb{H}^{n} \setminus \{ 0 \}$. Notice that $r_{+}(t)=\max_{\Sigma_{t}} r \rightarrow 0$ as $t \rightarrow -\infty$ by continuity. Hence $\Sigma_{t}$ is star-shaped and $Dr_{t}=0$ on $\Sigma_{t}$ from the above estimate, thus proving the assertion.
\end{proof}

\section{The Heintze-Karcher Inequality for Weak Mean Curvature}
In \cite{B13} and following earlier work by Heintze-Karcher \cite{HK78}, Brendle proved the following inequality for mean-convex domains in $\mathbb{H}^{n}$.

\begin{Theorem}[\cite{B13}, Theorem 3.5]
Let $\Omega_{0} \subset \mathbb{H}^{n}$ be a bounded domain with smooth $H >0$ boundary $\partial \Omega= \Sigma_{0}$. Then we have the integral inequality
    \begin{equation} \label{brendle_inequality}
    (n-1) \int_{\Sigma_{0}} \frac{f}{H} d\sigma \geq n \int_{\Omega_{0}} f d\Omega.
\end{equation}
\end{Theorem}
The proof of the Minkowski inequality for star-shaped hypersurfaces by Brendle-Hung-Wang crucially relies on \eqref{brendle_inequality}, and so a proof involving weak IMCF requires this as well. The goal of this section is to suitably extend \eqref{brendle_inequality} to a proper weak solution $\{ \Omega_{t} \}_{t \in (0,\infty)}$ of IMCF. Our approach here is an approximation by mean-convex mean curvature flow used in \cite{LW17} and \cite{HI08}.

\begin{Theorem}[\cite{LW17}, Lemma 4.2, c.f also \cite{HI08}, Lemma 2.6] \label{mcf_approximation}
Let $\Sigma_{0}=X_{0}(\Sigma^{n-1}) \subset \mathbb{H}^{n}$ be an immersion of class $C^{1}$ with bounded and non-negative measurable weak mean curvature $H$. Then $\Sigma_{0}$ is of class $C^{1,\alpha} \cap W^{2,p}$, and there exists a solution $X: \Sigma^{n-1} \times (0,\epsilon_{0}) \rightarrow (\mathbb{H}^{n},g_{\mathbb{H}^{n}})$ to mean curvature flow (MCF),

\begin{eqnarray} \label{mcf}
    \frac{\partial}{\partial \epsilon} X(x,\epsilon) = -H \nu(x,\epsilon), \hspace{2cm} (x,\epsilon) \in \Sigma^{n-1} \times (0, \epsilon_{0}),
\end{eqnarray}
converging to $\Sigma_{0}$ locally in $C^{1,\alpha} \cap W^{2,p}$, $\alpha \in (0,1), p \in (1,\infty)$ and with  $H(x,\epsilon) \rightarrow H(x)$ strongly in $L^{p}$ for $p \in [1,\infty)$ as $\epsilon \rightarrow 0$ \footnote{See the proof of \cite{HI08}, Lemma 2.6 on pages 447-448 for the convergence of $X_{\epsilon}$ and the $L^{p}$ convergence of $H_{\epsilon}$}. Furthermore, for every $\epsilon \in (0,\epsilon_{0})$ we have $\min_{\Sigma_{\epsilon}} H > 0$.
\end{Theorem}
Originally, Huisken-Ilmanen used this approximation to establish weak flow smoothing in $\mathbb{R}^{n}$, but here it is also useful for extending inequality \eqref{brendle_inequality} to $C^{1}$ hypersurfaces $\Sigma_{0}$ with non-negative bounded weak mean curvature. We study the reciporical mean curvature $H^{-1}_{\epsilon}$ of $\Sigma_{\epsilon}$ under MCF as $\epsilon \searrow 0$. $|| H_{\epsilon}^{-1}||_{L^{1}(\Sigma_{\epsilon})}$ is monotonically decreasing in $\epsilon$ under MCF, and this suggests $L^{1}$ convergence as $\epsilon \searrow 0$. However, we must handle the case where $H^{-1}$ is integrable and $\text{essinf}_{\Sigma_{0}} H = 0$ delicately.
\begin{Proposition} \label{approx_liminf}
Let $\Sigma_{0}$, $\Sigma_{\epsilon}$ be as in Theorem \ref{mcf_approximation}, and suppose that inverse weak mean curvature $H(x,0)^{-1}$ is $L^{1}$ integrable. Then there is a subsequence $\epsilon_{j} \searrow 0$ of times such that the inverse mean curvature $H(x,\epsilon)^{-1}$ of $\Sigma_{\epsilon_{j}}$ converges to $H(x,\epsilon)^{-1}$ strongly in $L^{1}$ as $\epsilon_{j} \searrow 0$.
\end{Proposition}
\begin{proof}

For each $k \in \mathbb{N}$, we consider the quantity

\begin{equation} \label{beta}
    \psi_{k}(\epsilon) =\int_{\Sigma_{\epsilon}} \frac{1}{e^{(n-1)\epsilon} H + k^{-1}} d\sigma.
\end{equation}
under \eqref{mcf}. The variation formula for mean curvature under MCF in $\mathbb{H}^{n}$ is

\begin{eqnarray*}
    \frac{\partial}{\partial \epsilon} H &=& \Delta_{\Sigma_{\epsilon}} H + \left( |h|^{2} + \text{Ric}(\nu,\nu) \right) H \label{mcf_variation}  \\
    &=& \Delta_{\Sigma_{\epsilon}} H+ \left( |h|^{2} - (n-1) \right) H, \nonumber
\end{eqnarray*}
and so

\begin{equation*}
    \frac{\partial}{\partial \epsilon} e^{(n-1)\epsilon} H = \Delta_{\Sigma_{\epsilon}} e^{(n-1)\epsilon} H + |h|^{2}e^{(n-1)\epsilon}H.
\end{equation*}
Using this, we compute the evolution of \eqref{beta} under MCF:
\begin{eqnarray} \label{decreasing}
   \frac{\partial}{\partial \epsilon} \psi_{k}(\epsilon) &=& \int_{\Sigma_{\epsilon}} -\frac{1}{\left( e^{(n-1)\epsilon} H + k^{-1} \right)^{2}} \frac{\partial}{\partial \epsilon} \left( e^{(n-1)\epsilon} H \right) - \frac{1}{e^{(n-1)\epsilon} H + k^{-1}} H^{2} d\sigma \nonumber \\
   &\leq& -\int_{\Sigma_{\epsilon}} \frac{2e^{2(n-1)\epsilon}}{\left( e^{(n-1)\epsilon} H + k^{-1} \right)^{3}} |\nabla H|^{2} d\sigma \leq 0,
\end{eqnarray}
and so \eqref{beta} is monotonically decreasing in $\epsilon$. By the $L^{p}$ convergence $H(x,\epsilon) \rightarrow H(x,0)$ and the bound $\phi_{k}(x) \leq k$, we have that
\begin{equation*}
   \psi_{k}(\epsilon_{j}) \rightarrow  \psi_{k}(0)
\end{equation*}
along a subsequence of times $\epsilon_{j}$ converging to $0$ by the dominated convergence theorem. Combining this with \eqref{decreasing} yields

\begin{equation} \label{regularized_inverse}
    \int_{\Sigma_{\epsilon}} \frac{1}{e^{(n-1)\epsilon} H + k^{-1}} d\sigma \leq  \int_{\Sigma_{0}} \frac{1}{H + k^{-1}} d\sigma, \hspace{1cm} k \in \mathbb{N}, \hspace{1cm} \epsilon \in (0,\epsilon_{0}).
\end{equation}
Now, for a fixed $\epsilon \in (0,\epsilon_{0})$ we take the $k \searrow \infty$ limit of each side of \eqref{regularized_inverse}. By monotone convergence, we get

\begin{equation*}
    \int_{\Sigma_{\epsilon}} \frac{1}{H} d\sigma \leq e^{(n-1)\epsilon} \int_{\Sigma_{0}} \frac{1}{H} d\sigma.
\end{equation*}
Hence

\begin{eqnarray*}
    \limsup_{\epsilon \searrow 0} \int_{\Sigma_{\epsilon}} \frac{1}{H} d\sigma \leq \int_{\Sigma_{0}} \frac{1}{H} d\sigma,
\end{eqnarray*}
and once again by $L^{p}$ convergence of $H_{\epsilon}$ we have that that $H_{\epsilon_{j}}^{-1} \rightarrow H^{-1}$ in $L^{1}$ along a subsequence $\epsilon_{j}$.

\end{proof}
Given this convergence, we may pass inequality \eqref{brendle_inequality} off to $C^{1}$ $\Sigma_{0}$ with bounded, non-negative $H$.

\begin{Theorem}[Heintze-Karcher Inequality for Weak Mean Curvature] \label{brendle_weak}
Let $\Omega_{0} \subset \mathbb{H}^{n}$ be a bounded domain. Suppose that $\partial \Omega_{0}= \Sigma_{0}=X_{0}(\Sigma)$ is $C^{1}$ with bounded and non-negative measurable weak mean curvature. Then 

\begin{equation}
    (n-1) \int_{\Sigma_{0}} \frac{f}{H} d \mu \geq n \int_{\Omega_{0}} f d\Omega,
\end{equation}
where $H$ is the weak mean curvature of $\Sigma_{0}$.
\end{Theorem}
\begin{proof}
 We assume $H^{-1}$ is integrable, since otherwise the statement is trivial. Let $X: \Sigma \times (0,\epsilon_{0}) \rightarrow \mathbb{H}^{n}$ be the approximating flow from Proposition \ref{mcf_approximation}, and denote by $\Omega_{\epsilon}$ the region bounded by $\Sigma_{\epsilon}$. Given Proposition \ref{approx_liminf}, we consider the subsequence $\epsilon_{j} \searrow 0$ from Proposition \ref{approx_liminf}. By $C^{0}$ convergence of $\Sigma_{\epsilon_{j}}$ and $L^{1}$ convergence of $H_{\epsilon_{j}}^{-1}$, we have

\begin{eqnarray*}
   \lim_{j} \int_{\Omega_{\epsilon_{j}}} f d\Omega &=& \int_{\Omega_{0}} f d\Omega, \\
   \lim_{j} \int_{\Sigma_{\epsilon_{j}}} \frac{f}{H} d\sigma &=& \int_{\Sigma_{0}} \frac{f}{H} d\sigma.
\end{eqnarray*}
Since $\Sigma_{\epsilon}$ is smooth with $H>0$, we apply the inequality \eqref{brendle_inequality} over each $\Omega_{\epsilon_{j}}$ to make the conclusion. 
\end{proof}
\section{The Volumetric Minkowski Inequality}
In this section, we prove the Minkowski inequality \eqref{minkowski} of Brendle-Hung-Wang for outer-minimizing domains of $\mathbb{H}^{n}$, $3 \leq n \leq 7$ using weak IMCF. Huisken introduced this approach in his lecture \cite{H16}, which was pursued by Freire-Schwartz \cite{FS13} in $\mathbb{R}^{n}$ and later Wei \cite{W18} in Schwarzschild space. Throughout this section, we will use the relation

\begin{equation} \label{hess}
    \nabla^{2} f = f g_{\mathbb{H}^{n}}
\end{equation}
for the Hessian of $f=\cosh\left(r \right)$ in $\mathbb{H}^{n}$. We also define the \textit{support function} $\phi: \Sigma_{0} \rightarrow \mathbb{R}$ of a $C^{1}$ hypersurface as
\begin{eqnarray} \label{support}
    \phi(x) &=& \langle \nabla f, \nu(x) \rangle.
\end{eqnarray}
\eqref{hess} implies for a smooth hypersurface $\Sigma_{0}^{n-1} \subset \mathbb{H}^{n}$ that

\begin{equation} \label{surface_potential}
    H \phi + \Delta_{\Sigma_{0}} f = (n-1) f,
\end{equation}
and for the domain $\Omega_{0}$ enclosed by $\Sigma_{0}$ that

\begin{equation} \label{div}
    \int_{\Sigma_{0}} \phi d\sigma = \int_{\Omega_{0}} \Delta_{\mathbb{H}^{n}} fd \mu = n \int_{\Omega_{0}} f d\Omega
\end{equation}
via the Divergence Theorem.

In \cite{BHW12}, Brendle-Hung-Wang considered the integral quantity

\begin{equation} \label{Q}
    Q(t) =  |\Sigma_{t}|^{\frac{2-n}{n-1}} \left( \int_{\Sigma_{t}} fH d\sigma - n(n-1) \int_{\Omega_{t}} f d\Omega  \right)
\end{equation}
under smooth IMCF in $\mathbb{H}^{n}$. To determine the evolution of $Q(t)$, recall the variation formula for mean curvature with $\text{Ric}(\nu,\nu)=-(n-1)$,

\begin{equation*}
    \frac{\partial}{\partial t} H = - \Delta_{\Sigma_{t}} \frac{1}{H} + \left((n-1) - |h|^{2}\right) \frac{1}{H}. 
\end{equation*}
Now, we compute
\begin{eqnarray*}
    \frac{\partial}{\partial t} Q(t) &=& \frac{2-n}{n-1} |\Sigma_{t}|^{\frac{2-n}{n-1}} \left( \int_{\Sigma_{t}} fH d\sigma - n(n-1) \int_{\Omega_{t}} f d\Omega  \right) \\
    & &+ |\Sigma_{t}|^{\frac{2-n}{n-1}} \left( \int_{\Sigma_{t}} \left( \frac{\partial}{\partial t}f\right)H + f \left(\frac{\partial}{\partial t} H \right) + fH d\sigma - n(n-1) \int_{\Sigma_{t}} \frac{f}{H} d\sigma  \right) \\
    &= & \frac{2-n}{n-1} |\Sigma_{t}|^{\frac{2-n}{n-1}} \left( \int_{\Sigma_{t}} fH d\sigma - n(n-1) \int_{\Omega_{t}} f d\Omega  \right) \\
    & & + |\Sigma_{t}|^{\frac{2-n}{n-1}} \left( \int_{\Sigma_{t}} \phi - f \Delta_{\Sigma_{t}} \frac{1}{H} - \frac{|h|^{2}}{H} f + (n-1) \frac{f}{H} + f H d\sigma - n(n-1) \int_{\Sigma_{t}} \frac{f}{H} d\sigma \right) \\
    &=& \frac{2-n}{n-1} |\Sigma_{t}|^{\frac{2-n}{n-1}} \left( \int_{\Sigma_{t}} fH d\sigma - n(n-1) \int_{\Omega_{t}} f d\Omega  \right) \\
    & & + |\Sigma_{t}|^{\frac{2-n}{n-1}} \left( \int_{\Sigma_{t}} 2 \phi - ( H\phi + \Delta_{\Sigma_{t}} f) \frac{1}{H}  + \frac{n-2}{n-1} f H - \frac{|\mathring{h}|^{2}}{H} f - (n-1)^{2} \frac{f}{H} d\sigma \right)\\
    &=& |\Sigma_{t}|^{\frac{2-n}{n-1}} \left( -\int_{\Sigma_{t}} \frac{|\mathring{h}|^{2}}{H}f d\sigma + n \left(-(n-1) \int_{\Sigma_{t}} \frac{f}{H} d\sigma + n \int_{\Omega_{t}} f d\Omega \right)  \right) \\
    &\leq& -|\Sigma_{t}|^{\frac{2-n}{n-1}} \int_{\Sigma_{t}} \frac{|\mathring{h}|^{2}}{H}f d\sigma \leq 0.
\end{eqnarray*}
Here, the third equation follows from \eqref{div} and \eqref{surface_potential} and the inequality follows from the Heintze-Karcher inequality \eqref{brendle_inequality}. The lower bound $\liminf_{t \nearrow \infty} Q(t) \geq (n-1) w_{n-1}^{\frac{1}{n-1}}$ arises from a Sobolev inequality, see Section 5 of \cite{BHW12}. 

To extend \eqref{minkowski_area}, we must show that $Q(t)$ remains monotone under weak IMCF. We remark that geometric flows have recently been used to prove families of inequalities in $\mathbb{H}^{n}$ related to the one in \cite{BHW12}-- see for example \cite{LWX14}, \cite{AW18}, \cite{LWX14}, \cite{HLW20}, \cite{KWWV21}, and \cite{WWZ23}-- but the flows used for these are not known to admit variational formulations, and so a weak flow approach to these inequalities is currently out of reach. In order to show monotonicity of $Q(t)$ under weak IMCF, we need to estimate the bulk term and the surface integral term in \eqref{Q} separately. First we determine the growth of $\int_{\Omega_{t}} f d\Omega$ along the proper weak solution. This may be accomplished using the robust version of the co-area formula from \cite{S18}.




\begin{Proposition}[Growth of the Bulk Term] \label{bulk}
Let $u \in C^{0,1}_{\text{loc}}(\mathbb{H}^{n})$, $3 \leq n \leq 7$, be a proper weak IMCF with initial condition $\Omega_{0}$. Then for every $0 \leq t_{1} < t_{2}$, we have for the potential function $f(r)=\cosh(r)$ that

\begin{equation} \label{volume_growth}
   \int_{\Omega_{t_{2}}} f d\Omega  \geq \int_{\Omega_{t_{1}}} f d\Omega + \frac{n}{n-1} \int_{t_{1}}^{t_{2}} \int_{\Omega_{s}} f d\Omega ds.
\end{equation}
\end{Proposition}
\begin{proof}
Given that $u \in C^{0,1}_{\text{loc}}(\mathbb{H}^{n} \setminus \Omega_{0})$, $\nabla u(x)$ is defined for 
 $\mathcal{H}^{n}$ a.e. $x \in \mathbb{H}^{n} \setminus \Omega_{0}$. From this, we define the $\mathcal{H}^{n}$-measurable function $h: \mathbb{H}^{n} \setminus \Omega_{0} \rightarrow \mathbb{R}$ by

\begin{equation*}
    h(x) = \begin{cases} 0 & \nabla u(x) =0, \\
   f(x)|\nabla u|^{-1} & \nabla u(x) \neq 0.
   \end{cases} 
\end{equation*}
Now, we apply the co-area formula for Lipschitz functions from Theorem 2.7.3 and Remark 2.7.4 in \cite{S18} and obtain

\begin{eqnarray*}
    \int_{\Omega_{t_{2}}} f d\Omega - \int_{\Omega_{t_{1}}} f d\Omega  \geq \int_{\Omega_{t_{2}} \setminus \Omega_{t_{1}}} h |\nabla u| d\Omega = \int_{t_{1}}^{t_{2}} \int_{\Sigma_{s}} h d\sigma ds = \int_{t_{1}}^{t_{2}} \int_{\Sigma_{s}} \frac{f}{H} d\sigma ds.
\end{eqnarray*}
Here, we used the fact that $\Sigma_{t} = \Sigma^{+}_{t} = \{ u = t\}$ for a.e. $t \in (0,\infty)$. Likewise, the last equation follows from the fact that for a.e. $t \in (0,\infty)$ we have $|\nabla u(x)| = H(x)$ for $\mathcal{H}^{n-1}$ a.e. $x \in \Sigma_{t}$. The conclusion follows from Theorem \ref{brendle_weak}.
\end{proof}

Note that Proposition \ref{bulk} is equivalent to

\begin{equation} \label{volume_growth2}
    \int_{\Omega_{t_{2}}} f d\Omega  \geq e^{\frac{n}{n-1} (t_{2}-t_{1})} \int_{\Omega_{t_{1}}} f d\Omega
\end{equation}
for $t_{2} > t_{1}$ by Gronwall's lemma-- we will use this explicit lower bound on the bulk term in Section 6.

Next, we estimate the surface integral term of \eqref{Q}, and for this it is more convenient to work with the smooth approximating functions $u_{\epsilon}$ from \eqref{u_epsilon}. Our proof here is similar to the one in \cite{W18}, Section 4, but with a different underlying potential equation. We first need an inequality for the regular level sets of $u^{i}=u_{\epsilon_{i}}$.
\begin{Lemma}
Let $\Omega_{0} \subset \mathbb{H}^{n}$ be a smooth bounded domain, and let $u \in  C^{\infty}(\mathbb{H}^{n} \setminus \overline{\Omega}_{0})$ a proper smooth function with $u|_{\Sigma_{0}} \equiv 0$ and $\mathcal{H}^{n}\left(\text{crit}( u ) \right) = 0$. Define the domain $\Omega_{t} = \{ u \leq t \} \subset \mathbb{H}^{n}$, and let $\Psi: (0,t) \rightarrow \mathbb{R}^{+}$ be a compactly-supported Lipschitz function. Then for the function
\begin{equation*}
    \psi (x) = (\Psi\circ u)(x), \hspace{2cm} x \in \Omega_{t},
\end{equation*}
 we have the integral identity

 \begin{equation} \label{integral_identity}
     - \int_{\Omega_{t}} f H \frac{\partial \psi }{\partial \nu} d\Omega = \int_{\Omega_{t}} \psi [2H \phi_{s} + f(H^{2} - |h|^{2})] d \Omega,
 \end{equation}
 where $\nu$, $H$, $h$, and $\phi$ are the unit normal, mean curvature, second fundamental form, and support function of the regular level sets $\Sigma_{s}=\{ u = s\}$ of $u$.
\end{Lemma}
\begin{proof}
We will initially assume $\Psi\in C^{1}_{0}(0,t)$. By Sard's theorem, a.e. level set $\Sigma_{s}= \{ u = s \}$, $s \in [0,t)$, is regular. For a regular level set $\Sigma_{s}$, we consider the variation field

\begin{equation*}
    \frac{\nabla u}{|\nabla u|^{2}} = |\nabla u|^{-1} \nu.
\end{equation*}
Taking the first variation formula for mean curvature with $\text{Ric}(\nu,\nu)=-(n-1)$, we find

\begin{equation*}
    -|\nabla u|^{-1} \frac{\partial H}{\partial \nu} = \Delta_{\Sigma_{s}} |\nabla u|^{-1} + (|h|^{2} -(n-1)) |\nabla u|^{-1}.
\end{equation*}
Multiplying by $f$ and integrating yields

\begin{eqnarray} \label{level_set_u}
    \int_{\Sigma_{s}} -f |\nabla u|^{-1} \frac{\partial H}{\partial \nu} d\sigma &=& \int_{\Sigma_{s}} f \Delta_{\Sigma_{s}} |\nabla u|^{-1} + (|h|^{2} -(n-1)) f |\nabla u|^{-1} d\sigma \nonumber \\
    &=& \int_{\Sigma_{s}} (\Delta_{\Sigma_{s}} f -(n-1)f +|h|^{2}f) |\nabla u|^{-1} d\sigma \\
    &=& \int_{\Sigma_{s}} - H \phi_{s} |\nabla u|^{-1} + |h|^{2}f |\nabla u|^{-1} d\sigma, \nonumber 
\end{eqnarray}
where we used the potential equation \eqref{surface_potential}. Since $\Sigma_{s}= \{ u = s \}$ is regular for a.e. $s \in (0,t)$, we apply \eqref{level_set_u} to the co-area formula

\begin{eqnarray} \label{H_deriv_identity}
    \int_{\Omega_{t}} f \psi \frac{\partial H}{\partial \nu} d\Omega &=& \int_{0}^{t} \int_{\Sigma_{s}} f \psi \frac{\partial H}{\partial \nu} |\nabla u|^{-1} d\sigma ds = \int_{0}^{t} \Psi(s) \int_{\Sigma_{s}} f \frac{\partial H}{\partial \nu} |\nabla u|^{-1} d\sigma ds \nonumber \\
    &=& \int_{0}^{t} \Psi(s) \int_{\Sigma_{s}} H \phi_{s} |\nabla u|^{-1} - f|h|^{2} |\nabla u|^{-1} d\sigma ds \\
    &=& \int_{\Omega_{t}} \psi (H \phi_{s} - f|h|^{2}) d\Omega, \nonumber
\end{eqnarray}
Now,

\begin{eqnarray} \label{divergence_formula}
    \text{div}( f \psi  H \nu) &=& \psi  H \phi + f H \frac{\partial \psi }{\partial \nu} + f \psi \frac{\partial H}{\partial \nu} + f\psi  H \text{div} \nu
\end{eqnarray}
Because $\Psi$ is compactly supported, $\psi $ vanishes over $\partial \Omega_{t}= \{ u = t\} \cup \partial M$. Therefore, the integral of the left-hand side of \eqref{divergence_formula} vanishes, and we are left with

\begin{eqnarray} \label{final_identity}
    -\int_{\Omega_{t}} fH \frac{\partial \psi }{\partial \nu} d\Omega &=& \int_{\Omega_{t}} \psi H \phi_{s} + f \psi \frac{\partial H}{\partial \nu} + f\psi H \text{div} \nu d\Omega \\
    &=& \int_{\Omega_{t}} \psi ( 2H \phi_{s} - f|h|^{2}) d\Omega + \int_{\Omega_{t}} f\psi H \text{div} \nu d\Omega. \nonumber
\end{eqnarray}
For the divergence term, we once again see from the co-area formula
\begin{eqnarray*}
    \int_{\Omega_{t}} f\psi H \text{div} \nu d\Omega &=& \int_{0}^{t} \int_{\Sigma_{s}} f \psi H \text{div}(\nu) |\nabla u|^{-1} d\sigma ds \\
    &=& \int_{0}^{t} \int_{\Sigma_{s}} f \psi H \text{div}_{\Sigma_{s}}(\nu) |\nabla u|^{-1} d\sigma ds \\
    &=& \int_{0}^{t} \int_{\Sigma_{s}} f \psi H^{2} |\nabla u|^{-1} d\sigma ds \\
    &=& \int_{\Omega_{t}} f\psi H^{2} d\Omega.
\end{eqnarray*}
Substituting this into \eqref{final_identity} yields \eqref{integral_identity}.
\end{proof}
\begin{Remark}
The formula

\begin{equation*}
    \int_{\Omega} f d\Omega = \int_{0}^{t} \int_{\{ u = s\}} \frac{f}{|\nabla u|} d\sigma ds
\end{equation*}
for $f \in C^{\infty}_{+}(\mathbb{H}^{n})$ and locally Lipschitz $u$ is sometimes applied erroneously in the literature. The above formula is always true \textit{as an inequality}, such as in the proof of Proposition \ref{bulk}. For exact identity, the critical set of $u$ must have measure $0$. This is indeed the case for $u^{i}$ solving \eqref{u_epsilon} (check the Laplacian at critical points). This is clarified in \cite{CF17}.
\end{Remark}

\eqref{integral_identity} implies an integral inequality for the level sets of a smooth, proper $u$. We may pass this inequality to level sets of a proper weak IMCF, which yields a growth estimate for the first term in \eqref{Q}.
\begin{Lemma}[Growth of the Total Mean Curvature] 
Let $u \in C^{0,1}_{\text{loc}}(\mathbb{H}^{n})$ be a proper weak IMCF with initial condition $\Omega_{0}$. Then for any $0 < t_{1} < t_{2}$ we have
\begin{equation} \label{key_inequality}
    \int_{\Sigma_{t_{2}}} fH d\sigma \leq \int_{\Sigma_{t_{1}}} fH d\sigma + \int_{t_{1}}^{t_{2}} \int_{\Sigma_{s}} \left( \frac{n-2}{n-1} fH + 2 \phi_{s} \right) d\sigma ds .
\end{equation}
\end{Lemma}

\begin{proof}
For the approximating sequence $u_{\epsilon_{i}}=u^{i}$ in Theorem \ref{approximate}, we have by \eqref{integral_identity} and the inequality $H^{2} \leq (n-1) |h|^{2}$ that

\begin{equation*}
    -\int_{\Omega^{i}_{t}} fH_{i} \frac{\partial \psi_{i}}{\partial \nu_{i}} d\Omega \leq \int_{\Omega^{i}_{t}} \psi_{i} \left( 2 H_{i} \phi^{i}_{s} + \frac{n-2}{n-1} H_{i}^{2} \right) d\Omega.
\end{equation*}
Thus the convergence \eqref{bulk_convergence} implies

\begin{equation} \label{weak_imcf_inequality}
    -\int_{\Omega_{t}} fH \frac{\partial \psi }{\partial \nu} d\Omega \leq \int_{\Omega_{t}} \psi \left( 2H \phi_{s} + \frac{n-2}{n-1} fH^{2} \right) d\Omega,
\end{equation}
where $H$ and $\phi_{s}$ are the weak mean curvature and support function of the weak flow surfaces $\Sigma_{t} = \partial \{ u < t \}$. Applying \eqref{weak_imcf_inequality} along with the co-area formula, we find for any $\Psi\in C^{0,1}_{0}((0,t))$ that 

\begin{eqnarray*}
    -\int_{0}^{t} \Psi'(s) \int_{\Sigma_{s}} fH d\sigma_{s} ds &=& - \int_{\Omega_{t}} \frac{\partial \psi }{\partial \nu} fH d\Omega \\
    &\leq& \int_{\Omega_{t}} \psi \left( 2 \phi_{s} H + \frac{n-2}{n-1} fH^{2} \right) d\Omega \\
    &=& \int_{\Omega_{t}} \psi \left( 2 \phi_{s} + \frac{n-2}{n-1} fH \right) |\nabla u| d\Omega \\
    &=& \int_{0}^{t} \Psi(s) \int_{\Sigma_{s}} \left( \frac{n-2}{n-1} fH + 2 \phi_{s} \right) d\sigma ds
\end{eqnarray*}

for $\Psi\in C^{0,1}_{0}((0,t))$. Fix $0 \leq t_{1} < t_{2}$, and for $0 < \delta < \frac{1}{2} (t_{2}-t_{1}) $, choose

\begin{equation*}
    \Psi(s) = \begin{cases} 0 & s \in (0,t_{1})  \\
    \frac{1}{\delta} (s- t_{1}) & s \in (t_{1}, t_{1} + \delta) \\
    1 & s \in (t_{1} + \delta, t_{2} - \delta) \\
    \frac{1}{\delta} (t_{2}-s) & s \in (t_{2} - \delta, t_{2}) \\
    0 & s \in (t_{2},t).
    
    \end{cases}
\end{equation*}
Then 

\begin{equation*}
    \frac{1}{\delta} \int_{t_{2} - \delta}^{t_{2}} \int_{\Sigma_{s}} fH d\sigma ds - \frac{1}{\delta} \int_{t_{1}}^{t_{1}+\delta} \int_{\Sigma_{s}} fH d\sigma ds \leq \int_{t_{1}}^{t_{2}} \int_{\Sigma_{s}} \left( \frac{n-2}{n-1} fH + 2\phi_{s} \right) d\sigma ds.
\end{equation*}
The limit of the left-hand side as $\delta \rightarrow 0$ is well-defined for a.e. $0 < t_{1} < t_{2}$, and so for a.e. $0 < t_{1} < t_{2}$ we have

\begin{eqnarray}
    \int_{\Sigma_{t_{2}}} fH d\sigma - \int_{\Sigma_{t_{1}}} fH d\sigma &\leq& \int_{t_{1}}^{t_{2}} \int_{\Sigma_{s}} \left(\frac{n-2}{n-1} fH + 2\phi_{s} \right) d\sigma ds \label{fH_upper_bound}.
\end{eqnarray}
To extend this inequality to all positive times, for any $0 < t_{1} < t_{2}$ we choose sequences $t_{j} \nearrow t_{1}$ and $\tilde{t}_{j} \nearrow t_{2}$ such that \eqref{fH_upper_bound} holds for the times $t_{j}, \tilde{t}_{j}$. By the second property in Theorem \ref{imcf_properties}, we have that 

\begin{equation*}
    \Sigma_{t_{j}} \rightarrow \Sigma_{t_{1}} \hspace{2cm} \text{  in  } \hspace{1cm} C^{1,\alpha}, \\
\end{equation*}
resp. for $\tilde{t}_{j},t_{2}$, away from a small singular set $Z$. Equation (1.13) in \cite{HI99} then implies $L^{1}$ convergence of the weak mean curvature, and so altogether we find that

\begin{eqnarray*}
    \lim_{j \rightarrow \infty} \int_{\Sigma_{t_{j}}} fH d\sigma &=& \int_{\Sigma_{t_{1}}} fH d\sigma, \\
    \lim_{j \rightarrow \infty} \int_{\Sigma_{\tilde{t}_{j}}} fH d\sigma &=& \int_{\Sigma_{t_{2}}} fH d\sigma.
\end{eqnarray*}
Therefore, applying \eqref{fH_upper_bound} to $t_{j}, \tilde{t}_{j}$, we may conclude

\begin{eqnarray*}
    \int_{\Sigma_{t_{2}}} fH d\sigma - \int_{\Sigma_{t_{1}}} fH d\sigma &\leq& \lim_{j \rightarrow \infty} \int_{t_{j}}^{\tilde{t}_{j}} \int_{\Sigma_{s}} \left( \frac{n-2}{n-1}  fH + 2 \phi_{s}\right) d\sigma ds \\
    &=& \int_{t_{1}}^{t_{2}} \int_{\Sigma_{s}} \left( \frac{n-2}{n-1} fH + 2 \phi_{s} \right) d\sigma ds.
\end{eqnarray*}
For any $0 < t_{1} < t_{2}$.
\end{proof}
\begin{proof}[Proof of Theorem \ref{minkowski_ineq}, I]
Consider the quantity

\begin{equation*}
    q(t)= \int_{\Sigma_{t}} fH d\sigma - n(n-1) \int_{\Omega_{t}} f d\Omega.
\end{equation*}
Combining \eqref{key_inequality} and \eqref{volume_growth}, we find for $0 < t_{1} < t_{2}$ that

\begin{eqnarray} \label{q}
    q(t_{1}) - q(t_{2}) &\leq& \int_{t_{1}}^{t_{2}}\left( \frac{n-2}{n-1} \int_{\Sigma_{s}} f H d\sigma + 2 \int_{\Sigma_{s}} \phi_{s} d\sigma - n^{2} \int_{\Omega_{s}} f d\Omega  \right) ds  \\
    &=& \int_{t_{1}}^{t_{2}} \left( \frac{n-2}{n-1} \int_{\Sigma_{s}} fH d\sigma -n(n-2) \int_{\Omega_{s}} f d\Omega \right) ds = \frac{n-2}{n-1} \int_{t_{1}}^{t_{2}} q(s) ds, \nonumber
\end{eqnarray}
where we once again applied \eqref{div} over the $C^{1}$ domain $\Omega_{s}$. Therefore, if $\Sigma_{0}$ is outer-minimizing then the area growth formula from Theorem \ref{imcf_properties} and Gronwall's Lemma yield that $Q(t_{2}) \leq Q(t_{1})$ for $t_{2} > t_{1} > 0$. By Theorem 1.2 in \cite{BHW12}, \footnote{Theorem 1.2 in the ArXiv version} we have that

\begin{equation*}
    Q(t) \geq (n-1) w_{n-1}^{\frac{1}{n-1}} \hspace{1cm} \text{ for all} \hspace{1cm} t>T,
\end{equation*}
and so $Q(t) \geq (n-1)w_{n-1}^{\frac{1}{n-1}}$ for all positive $t$. By the $C^{1,\alpha}$ convergence $\Sigma_{t} \rightarrow \Sigma^{+}_{0}$ as $t \searrow 0$, this implies

\begin{equation*}
    |\Sigma_{0}^{+}|^{-\frac{n-2}{n-1}} \left( \int_{\Sigma_{0}^{+}} f H d\sigma - n(n-1) \int_{\Omega_{0}^{+}} f d\Omega \right) \geq (n-1) w_{n-1}^{\frac{1}{n-1}}.
\end{equation*}
Since  $H_{\Sigma_{0}^{+} \setminus \Sigma_{0}}=0$ by Property 1.4 (iii), we have that  $\int_{\Sigma_{0}} |fH| d\sigma \geq \int_{\Sigma_{0}^{+}} |fH| d\sigma$, thereby proving inequality \eqref{minkowski}.

It remains to address the case of equality. Suppose that $Q(0)=Q(t)$ for every $t$, and let
\begin{equation*}
    t_{0}= \inf \{ t \in (0,\infty) | \Sigma_{t} \text{  is star-shaped} \}
\end{equation*}
It follows once again by Theorem 1.2 in \cite{LW17} that $\{ \Sigma_{t} \}_{t \in (t_{0},\infty)}$ is a smooth solution to \eqref{IMCF}. By the equality statement in \cite{BHW12}, this means that $\Sigma_{t}$ is a slice $\mathbb{S}^{n-1} \times \{ r(t) \}$ of $\mathbb{H}^{n}$ for each $t \in (t_{0},\infty)$. If $Q(t)=Q(0)$ for each $t$, then as in Section 8 of \cite{HI99} we must have that $\Sigma_{t}=\Sigma^{+}_{t}$ for each $t \in (0,\infty)$, i.e. the flow has no jumps. The $C^{1}$ convergence $\Sigma_{t} \rightarrow \Sigma^{+}_{t_{0}}=\Sigma_{t_{0}} $ as $t \searrow t_{0}$ from Theorem \ref{imcf_properties}(ii) then implies that $\Sigma_{t_{0}}= \{ r(t_{0}) \} \times \mathbb{S}^{n-1}$ is also a slice of \eqref{hyperbolic_space}.

Now suppose that $t_{0} > 0$. Since $\Sigma_{t_{0}} = \{ r(t_{0}) \} \times \mathbb{S}^{n-1}$, we must have that the support function $\phi_{t_{0}}(x) = \sinh (r(t_{0}))$ is a positive constant. By the $C^{1}$ convergence $\Sigma_{t} \rightarrow \Sigma_{t_{0}}$ as $t \nearrow t_{0}$, we get that $\phi > 0$ on $\Sigma_{t}$ for $t \in (t_{0} - \epsilon,t_{0})$. This is a contradiction, and so we conclude that $t_{0}=0$ and hence $\Sigma_{0}$ is a slice of $\mathbb{H}^{n}$.
\end{proof}

\section{The Pure Area Minkowski Inequality}

In this Section, we extend the inequality of De Lima and Girao \cite{LG12} to outer-minimizing $\Omega_{0}$. In their paper, these authors considered the functional 

\begin{equation} \label{P(t)}
    P(t) = |\Sigma_{t}|^{\frac{2-n}{n-1}} \left( \int_{\Sigma_{t}} f H d\sigma - (n-1) w_{n-1}^{-\frac{1}{n-1}} |\Sigma_{t}|^{\frac{n}{n-1}} \right)
\end{equation}
under smooth star-shaped IMCF. Notice that if

\begin{equation} \label{volume_diff}
    \frac{n}{w_{n-1}} \int_{\Omega_{0}} f d\Omega \geq \left( \frac{|\Sigma_{0}|}{w_{n-1}} \right)^{\frac{n}{n-1}},
\end{equation}
then $P(0) \geq Q(0)$, and so in this case inequality \eqref{minkowski} implies \eqref{minkowski_area} for mean-convex, star-shaped $\Sigma_{0}$. In fact, condition \eqref{volume_diff} is preserved under smooth IMCF due to the Heintze-Karcher inequality, so it suffices to show that the quantity $P(t)$ is monotone non-increasing over the largest interval $(0,\tilde{t})$ on which

\begin{equation*}
    \frac{n}{w_{n-1}} \int_{\Omega_{t}} f d\Omega < \left(\frac{|\Sigma_{t}|}{w_{n-1}} \right)^{\frac{n}{n-1}}.
\end{equation*}
For $t$ in this interval, we compute

\begin{eqnarray*}
    \frac{\partial}{\partial t} P(t) &=& \frac{2-n}{n-1} |\Sigma_{t}|^{\frac{2-n}{n-1}} \left( \int_{\Sigma_{t}} f H d \mu - (n-1) w_{n-1}^{-\frac{1}{n-1}} |\Sigma_{t}|^{\frac{n}{n-1}} \right)  \\
    & & + |\Sigma_{t}|^{\frac{2-n}{n-1}} \left( \frac{\partial}{\partial t} \int_{\Sigma_{s}} fH d\sigma - n w_{n-1}^{-\frac{1}{n-1}} |\Sigma_{t}|^{\frac{n}{n-1}} \right) \\
    &=& \frac{2-n}{n-1} |\Sigma_{t}|^{\frac{2-n}{n-1}} \left( \int_{\Sigma_{t}} f H d \sigma - (n-1) w_{n-1}^{-\frac{1}{n-1}} |\Sigma_{t}|^{\frac{n}{n-1}} \right) \\
    & & + |\Sigma_{t}|^{\frac{2-n}{n-1}} \left(\int_{\Sigma_{t}} \frac{n-2}{n-1} fH - \frac{|\mathring{h}|^{2}}{H} f d\sigma + 2n \int_{\Omega_{t}}f d\Omega- n w_{n-1}^{-\frac{1}{n-1}} |\Sigma_{t}|^{\frac{n}{n-1}}  \right) \\
    &\leq& |\Sigma_{t}|^{\frac{2-n}{n-1}} \left( \int_{\Sigma_{t}} - \frac{|\mathring{h}|^{2}}{H} f d\sigma + 2n \int_{\Omega_{t}} f d\Omega - 2 w_{n-1}^{-\frac{1}{n-1}} |\Sigma_{t}|^{\frac{n}{n-1}}  \right) \leq 0.
\end{eqnarray*}
Suppose $\tilde{t} < +\infty$. Since $\tilde{t}$ satisfies \eqref{volume_diff} we have $P(0) \geq \widetilde{Q}(\tilde{t}) \geq Q(\tilde{t}) \geq (n-1)w_{n-1}^{\frac{1}{n-1}}$. De Lima and Girao carry out the asymptotic analysis for the case $\tilde{t}=+\infty$ in Appendix A of \cite{LG12}. Their result may be straight-forwardly extended using results from the previous sections. 

\begin{proof}[Proof of Theorem \ref{minkowski_ineq}, II]

Given an outer-minimizing initial domain $\Omega_{0} \subset \mathbb{H}^{n}$, we consider two cases separately:

$(i)$ $\frac{n}{w_{n-1}} \int_{\Omega_{0}} f d\Omega \geq \left(\frac{|\Sigma_{0}|}{w_{n-1}} \right)^{\frac{n}{n-1}}$: since \eqref{minkowski} holds on $\Omega_{0}$ according to the previous section, we conclude that \eqref{minkowski_area} holds as well.

$(ii)$ $\frac{n}{w_{n-1}} \int_{\Omega_{0}} f d\Omega < \left(\frac{|\Sigma_{0}|}{w_{n-1}} \right)^{\frac{n}{n-1}}$: first, observe that if
\begin{equation*}
    \frac{n}{w_{n-1}} \int_{\Omega_{t_{0}}} f d\Omega \geq \left(\frac{|\Sigma_{t_{0}}|}{w_{n-1}} \right)^{\frac{n}{n-1}}
\end{equation*}
for some $t_{0} > 0$, then the formula $|\Sigma_{t}|=e^{(t - t_{0})}|\Sigma_{t_{0}}|$ and the growth estimate \eqref{volume_growth2} on $\int_{\Omega_{t}} f d\Omega$ imply that
\begin{equation*}
    \frac{n}{w_{n-1}} \int_{\Omega_{t}} f d\Omega \geq \left(\frac{|\Sigma_{t}|}{w_{n-1}} \right)^{\frac{n}{n-1}}, \hspace{1cm} t \in [t_{0},\infty).
\end{equation*}
Therefore, we let $(0,\widetilde{t})$, $\widetilde{t} \leq \infty$, be the largest interval over which

\begin{equation*}
    \frac{n}{w_{n-1}} \int_{\Omega_{t}} f d\Omega < \left(\frac{|\Sigma_{t}|}{w_{n-1}} \right)^{\frac{n}{n-1}},
\end{equation*}
and we show that $P(t)$ is monotone non-increasing over $(0,\widetilde{t})$. As in the proof of the volumetric Minkowski inequality, for any $0 \leq t_{1} < t_{2} < \widetilde{t}$ we have

\begin{eqnarray}
    \int_{\Sigma_{t_{2}}} fH d\sigma - \int_{\Sigma_{t_{1}}} fH d\sigma &\leq& \int_{t_{1}}^{t_{2}} \int_{\Sigma_{s}} \left( 2 \phi_{s} + \frac{n-2}{n-1} fH \right) d\sigma ds \nonumber \\
    &=& \int_{t_{1}}^{t_{2}}  \left( 2n \int_{\Omega_{s}} f d\Omega + \int_{\Sigma_{s}} \frac{n-2}{n-1} fH d\sigma \right) ds \label{P(t)1} \\
    &<& \int_{t_{1}}^{t_{2}}\left( 2 w_{n-1}^{-\frac{1}{n-1}} |\Sigma_{s}|^{\frac{n}{n-1}} + \int_{\Sigma_{s}} \frac{n-2}{n-1} fH d\sigma \right) ds, \nonumber
\end{eqnarray}
where the last line comes from the volume assumption. On the other hand, we can write

\begin{equation}
(n-1) w_{n-1}^{-\frac{1}{n-1}} |\Sigma_{t_{2}}|^{\frac{n}{n-1}} - (n-1) w_{n-1}^{-\frac{1}{n-1}} |\Sigma_{t_{1}}|^{\frac{n}{n-1}} = \int_{t_{1}}^{t_{2}} n w_{n-1}^{-\frac{1}{n-1}} |\Sigma_{s}|^{\frac{n}{n-1}} ds. \label{P(t)2}
\end{equation}
Combining, we get for $P(t)= \int_{\Sigma_{t}} fH d\sigma - (n-1) w_{n-1}^{-\frac{1}{n-1}} |\Sigma_{t}|^{-\frac{n}{n-1}}$ that

\begin{eqnarray*}
    \widetilde{q}(t_{2}) - \widetilde{q}(t_{1}) &\leq& \int_{t_{1}}^{t_{2}} \left( (2-n) w_{n}^{-\frac{1}{n-1}} |\Sigma_{s}|^{\frac{n}{n-1}} + \int_{\Sigma_{s}} \frac{n-2}{n-1} fH d\sigma \right) ds \\
    &=& \frac{n-2}{n-1} \int_{t_{1}}^{t_{2}} \widetilde{q}(s) ds.
\end{eqnarray*} 
Using Gronwall's Lemma once again, we conclude that $\widetilde{Q}(t_{2}) \leq \widetilde{Q}(t_{1})$ for $P(t)= \left( |\Sigma_{t}|\right)^{\frac{2-n}{n-1}} P(t)$, and so $P$ is monotone over $(0,\widetilde{t})$. If $\widetilde{t} < +\infty$, then for each $t \in (0,\widetilde{t})$ we have

\begin{eqnarray*}
P(t) \geq P(\widetilde{t}) \geq Q(\widetilde{t}) \geq (n-1)w_{n-1}^{\frac{1}{n-1}}.
\end{eqnarray*}
If $\widetilde{t}= +\infty$, then since $\Sigma_{t}$ solves the smooth flow for $t > T$, the asymptotic situation reduces to the one in \cite{LG12}. Altogether, we conclude that

\begin{equation}
    P(t) \geq (n-1) w_{n-1}^{\frac{1}{n-1}}, \hspace{1cm} \text{for all} \hspace{1cm} t > 0.
\end{equation}
As before, $C^{1,\alpha}$ convergence then yields

\begin{equation*}
    \left( |\Sigma_{0}^{+}| \right)^{-\frac{n-2}{n-1}} \left( \int_{\Sigma_{0}^{+}} f H d\sigma - (n-1)w_{n-1}^{-\frac{1}{n-1}} |\Sigma_{0}^{+}|^{\frac{n}{n-1}} \right) \geq (n-1)w_{n-1}^{\frac{1}{n-1}}.
\end{equation*}
Both inequality \eqref{minkowski_area} and the rigidity follow from the same argument as the previous section.
\end{proof}

\section*{Conflict of Interest Statement}
On behalf of all authors, the corresponding author states that there is no conflict of interest.

\section*{Data Availability Statement}
Data sharing is not applicable to this article as no datasets were generated or analyzed during the current study.
\printbibliography[title={References}]

@article{HI99,
author = {Gerhard Huisken and Tom Ilmanen},
title = {{The Inverse Mean Curvature Flow and the Riemannian Penrose Inequality}},
volume = {59},
journal = {Journal of Differential Geometry},
number = {3},
publisher = {Lehigh University},
pages = {353 -- 437},
year = {2001},
doi = {10.4310/jdg/1090349447},
URL = {https://doi.org/10.4310/jdg/1090349447}
}

@article{HI08,
author = {Gerhard Huisken and Tom Ilmanen},
title = {{Higher regularity of the inverse mean curvature flow}},
volume = {80},
journal = {Journal of Differential Geometry},
number = {3},
publisher = {Lehigh University},
pages = {433 -- 451},
year = {2008},
doi = {10.4310/jdg/1226090483},
URL = {https://doi.org/10.4310/jdg/1226090483}
}

@article{CG01,
author = {Chow, Bennett and Gulliver, Robert},
year = {2001},
month = {04},
pages = {261-280},
title = {Aleksandrov reflection and geometric evolution of hypersurfaces},
volume = {9},
journal = {Communications in Analysis and Geometry},
doi = {10.4310/CAG.2001.v9.n2.a2}
}

@article{LW17,
   title={On inverse mean curvature flow in Schwarzschild space and Kottler space},
   volume={56},
   ISSN={1432-0835},
   url={http://dx.doi.org/10.1007/s00526-017-1160-6},
   DOI={10.1007/s00526-017-1160-6},
   number={3},
   journal={Calculus of Variations and Partial Differential Equations},
   publisher={Springer Science and Business Media LLC},
   author={Li, Haizhong and Wei, Yong},
   year={2017} }

@article{W18,
   title={On the Minkowski-type inequality for outward minimizing hypersurfaces in Schwarzschild space},
   volume={57},
   ISSN={1432-0835},
   url={http://dx.doi.org/10.1007/s00526-018-1342-x},
   DOI={10.1007/s00526-018-1342-x},
   number={2},
   journal={Calculus of Variations and Partial Differential Equations},
   publisher={Springer Science and Business Media LLC},
   author={Wei, Yong},
   year={2018} }

@article{BHW12,
    author = "Brendle, Simon and Hung, Pei-Ken and Wang, Mu-Tao",
    title = "{A Minkowski Inequality for Hypersurfaces in the Anti-de Sitter-Schwarzschild Manifold}",
    eprint = "1209.0669",
    archivePrefix = "arXiv",
    primaryClass = "math.DG",
    doi = "10.1002/cpa.21556",
    journal = "Commun. Pure Appl. Math.",
    volume = "69",
    pages = "124--144",
    year = "2016"
}

@article{LG12,
author = {De Lima, Levi and Girão, Frederico},
year = {2012},
month = {09},
pages = {},
title = {An Alexandrov–Fenchel-Type Inequality in Hyperbolic Space with an Application to a Penrose Inequality},
volume = {17},
journal = {Annales Henri Poincaré},
doi = {10.1007/s00023-015-0414-0}
}

@article{H19,
   title={Inverse mean curvature flow over non-star-shaped surfaces},
   volume={29},
   ISSN={1945-001X},
   url={http://dx.doi.org/10.4310/MRL.2022.v29.n4.a7},
   DOI={10.4310/mrl.2022.v29.n4.a7},
   number={4},
   journal={Mathematical Research Letters},
   publisher={International Press of Boston},
   author={Harvie, Brian},
   year={2022},
   pages={1065–1086} }

@article{G11,
author = {Claus Gerhardt},
title = {{Inverse curvature flows in hyperbolic space}},
volume = {89},
journal = {Journal of Differential Geometry},
number = {3},
publisher = {Lehigh University},
pages = {487 -- 527},
year = {2011},
doi = {10.4310/jdg/1335207376},
URL = {https://doi.org/10.4310/jdg/1335207376}
}

@article{SZ21,
author = {Shi, Yuguang and Zhu, Jintian},
year = {2021},
month = {05},
pages = {},
title = {Regularity of inverse mean curvature flow in asymptotically hyperbolic manifolds with dimension 3},
volume = {64},
journal = {Science China Mathematics},
doi = {10.1007/s11425-020-1860-6}
}

@article{S17,
title = {The inverse mean curvature flow in warped cylinders of non-positive radial curvature},
journal = {Advances in Mathematics},
volume = {306},
pages = {1130-1163},
year = {2017},
issn = {0001-8708},
doi = {https://doi.org/10.1016/j.aim.2016.11.003},
url = {https://www.sciencedirect.com/science/article/pii/S0001870816315110},
author = {Julian Scheuer}
}

@article{L16,
  title={Inverse curvature flow in anti-de Sitter-Schwarzschild manifold},
  author={Siyuan Lu},
  journal={Communications in Analysis and Geometry},
  year={2016},
  url={https://api.semanticscholar.org/CorpusID:119332350}
}

@article{DGS12,
   title={Penrose Type Inequalities for Asymptotically Hyperbolic Graphs},
   volume={14},
   ISSN={1424-0661},
   url={http://dx.doi.org/10.1007/s00023-012-0218-4},
   DOI={10.1007/s00023-012-0218-4},
   number={5},
   journal={Annales Henri Poincaré},
   publisher={Springer Science and Business Media LLC},
   author={Dahl, Mattias and Gicquaud, Romain and Sakovich, Anna},
   year={2012},
 pages={1135–1168} }

@misc{L10,
      title={The Graphs Cases of the Riemannian Positive Mass and Penrose Inequalities in All Dimensions}, 
      author={Mau-Kwong George Lam},
      year={2010},
      eprint={1010.4256},
      archivePrefix={arXiv},
      primaryClass={math.DG}
}

@article{N10,
author = {Andr{\'e} Neves},
title = {{Insufficient convergence of inverse mean curvature flow on asymptotically hyperbolic manifolds}},
volume = {84},
journal = {Journal of Differential Geometry},
number = {1},
publisher = {Lehigh University},
pages = {191 -- 229},
year = {2010},
doi = {10.4310/jdg/1271271798},
URL = {https://doi.org/10.4310/jdg/1271271798}
}

@article{HW14,
author = {Hung, Pei-Ken and Wang, Mu-Tao},
year = {2014},
month = {06},
pages = {},
title = {Inverse mean curvature flows in the hyperbolic 3-space revisited},
volume = {54},
journal = {Calculus of Variations and Partial Differential Equations},
doi = {10.1007/s00526-014-0780-3}
}

@article{A58,
author = {A.D. Alexandrov},
title= {A Uniqueness Theorem for Surfaces in the Large},
journal= {Vestnik Leningrad University: Mathematics},
volume= {13},
number={19},
pages={5-8},
year={1958}
}

@article{D11,
author = {Ding, Qi},
year = {2011},
month = {01},
pages = {27-44},
title = {The inverse mean curvature flow in rotationally symmetric spaces},
volume = {32},
journal = {Chin. Ann. Math. Ser. B},
doi = {10.1007/s11401-010-0626-z}
}

@article{B13,
     author = {Brendle, Simon},
     title = {Constant mean curvature surfaces in warped product manifolds},
     journal = {Publications Math\'ematiques de l'IH\'ES},
     pages = {247--269},
     publisher = {Springer-Verlag},
     volume = {117},
     year = {2013},
     doi = {10.1007/s10240-012-0047-5},
     mrnumber = {3090261},
     zbl = {1273.53052},
     language = {en},
     url = {http://www.numdam.org/articles/10.1007/s10240-012-0047-5/}
}

@article{C23,
url = {https://doi.org/10.1515/ans-2022-0034},
title = {Aleksandrov reflection for extrinsic geometric flows of Euclidean hypersurfaces},
title = {},
author = {Bennett Chow},
pages = {20220034},
volume = {23},
number = {1},
journal = {Advanced Nonlinear Studies},
doi = {doi:10.1515/ans-2022-0034},
year = {2023},
lastchecked = {2024-02-21}
}

@article{HK78,
author = {Heintze, Ernst and Karcher, Hermann},
journal = {Annales scientifiques de l'École Normale Supérieure},
language = {eng},
number = {4},
pages = {451-470},
publisher = {Elsevier},
title = {A general comparison theorem with applications to volume estimates for submanifolds},
url = {http://eudml.org/doc/82023},
volume = {11},
year = {1978},
}

@article{ES91,
author = {L. C. Evans and J. Spruck},
title = {{Motion of level sets by mean curvature. I}},
volume = {33},
journal = {Journal of Differential Geometry},
number = {3},
publisher = {Lehigh University},
pages = {635 -- 681},
year = {1991},
doi = {10.4310/jdg/1214446559},
URL = {https://doi.org/10.4310/jdg/1214446559}
}

@article{CGG91,
author = {Yun Gang Chen and Yoshikazu Giga and Shun'ichi Goto},
title = {{Uniqueness and existence of viscosity solutions of generalized mean curvature flow equations}},
volume = {33},
journal = {Journal of Differential Geometry},
number = {3},
publisher = {Lehigh University},
pages = {749 -- 786},
year = {1991},
doi = {10.4310/jdg/1214446564},
URL = {https://doi.org/10.4310/jdg/1214446564}
}

@misc{S18,
  author        = {Leon Simon},
  title         = {Introduction to Geometric Measure Theory},
  year          = {2018},
}

@article{CF17,
author = {Cadeddu, Lucio and Farina, Maria},
year = {2017},
month = {06},
pages = {},
title = {A brief note on the coarea formula},
volume = {88},
journal = {Abhandlungen aus dem Mathematischen Seminar der Universität Hamburg},
doi = {10.1007/s12188-017-0183-4}
}

@article{HLW20,
   title={Locally constrained curvature flows and geometric inequalities in hyperbolic space},
   volume={382},
   ISSN={1432-1807},
   url={http://dx.doi.org/10.1007/s00208-020-02076-4},
   DOI={10.1007/s00208-020-02076-4},
   number={3–4},
   journal={Mathematische Annalen},
   publisher={Springer Science and Business Media LLC},
   author={Hu, Yingxiang and Li, Haizhong and Wei, Yong},
   year={2020},
   month=sep, pages={1425–1474} }

@article{LWX14,
   title={A geometric inequality on hypersurface in hyperbolic space},
   volume={253},
   ISSN={0001-8708},
   url={http://dx.doi.org/10.1016/j.aim.2013.12.003},
   DOI={10.1016/j.aim.2013.12.003},
   journal={Advances in Mathematics},
   publisher={Elsevier BV},
   author={Li, Haizhong and Wei, Yong and Xiong, Changwei},
   year={2014},
   month=mar, pages={152–162} }

@article{AW18,
   title={Quermassintegral preserving curvature flow in Hyperbolic space},
   volume={28},
   ISSN={1420-8970},
   url={http://dx.doi.org/10.1007/s00039-018-0456-9},
   DOI={10.1007/s00039-018-0456-9},
   number={5},
   journal={Geometric and Functional Analysis},
   publisher={Springer Science and Business Media LLC},
   author={Andrews, Ben and Wei, Yong},
   year={2018},
   month=jul, pages={1183–1208} }

@article{WWZ23,
   title={Shifted inverse curvature flows in hyperbolic space},
   volume={62},
   ISSN={1432-0835},
   url={http://dx.doi.org/10.1007/s00526-023-02429-2},
   DOI={10.1007/s00526-023-02429-2},
   number={3},
   journal={Calculus of Variations and Partial Differential Equations},
   publisher={Springer Science and Business Media LLC},
   author={Wang, Xianfeng and Wei, Yong and Zhou, Tailong},
   year={2023},
   month=jan }

@online{H16,
        title = {Marston Morse- Inverse Mean Curvature Flow and Isoperimetric Inequalities},
        date = {2016},
        author = {Gerhard Huisken},
        url = {https://www.youtube.com/watch?v=Qt09iMPUcYY},
        note= {Video lecture}
    }

@article{KWWV21,
   title={On an inverse curvature flow in two-dimensional space forms},
   volume={384},
   ISSN={1432-1807},
   url={http://dx.doi.org/10.1007/s00208-021-02285-5},
   DOI={10.1007/s00208-021-02285-5},
   number={1–2},
   journal={Mathematische Annalen},
   publisher={Springer Science and Business Media LLC},
   author={Kwong, Kwok-Kun and Wei, Yong and Wheeler, Glen and Wheeler, Valentina-Mira},
   year={2021},
   month=oct, pages={1–24} }

@article{FS13,
   title={Mass-Capacity Inequalities for Conformally Flat Manifolds with Boundary},
   volume={39},
   ISSN={1532-4133},
   url={http://dx.doi.org/10.1080/03605302.2013.851211},
   DOI={10.1080/03605302.2013.851211},
   number={1},
   journal={Communications in Partial Differential Equations},
   publisher={Informa UK Limited},
   author={Freire, Alexandre and Schwartz, Fernando},
   year={2013},
   month=dec, pages={98–119} }

@article{BFM24,
   title={Minkowski inequality on complete Riemannian manifolds with nonnegative Ricci curvature},
   volume={17},
   ISSN={2157-5045},
   url={http://dx.doi.org/10.2140/apde.2024.17.3039},
   DOI={10.2140/apde.2024.17.3039},
   number={9},
   journal={Analysis and PDE},
   publisher={Mathematical Sciences Publishers},
   author={Benatti, Luca and Fogagnolo, Mattia and Mazzieri, Lorenzo},
   year={2024},
   month=Nov, pages={3039–3077} }
\end{document}